\documentclass[letterpaper, 11pt,  reqno]{amsart}

\usepackage[margin=1.1in,marginparwidth=1.5cm, marginparsep=0.5cm]{geometry}

\usepackage{amsmath,amssymb,amscd,amsthm,amsxtra, esint, comment}

\usepackage{kotex}
\usepackage{mathrsfs} 

\usepackage[implicit=true]{hyperref}

\usepackage{color} 
\usepackage{ulem}
\usepackage[makeroom]{cancel}

\allowdisplaybreaks[2]

\definecolor{gr}{rgb}   {0.,   0.69,   0.23 }
\definecolor{bl}{rgb}   {0.,   0.5,   1. }
\definecolor{mg}{rgb}   {0.85,  0.,    0.85}
\definecolor{yl}{rgb}   {0.8,  0.7,   0.}
\definecolor{or}{rgb}  {0.7,0.2,0.2}

\newtheorem{theorem}{Theorem} [section]

\newtheorem{lemma}[theorem]{Lemma}
\newtheorem{proposition}[theorem]{Proposition}
\newtheorem{remark}[theorem]{Remark}

\DeclareMathOperator{\Id}{Id}

\newcommand{\I}{\hspace{0.5mm}\text{I}\hspace{0.5mm}}

\newcommand{\noi}{\noindent}
\newcommand{\Z}{\mathbb{Z}}
\newcommand{\R}{\mathbb{R}}

\newcommand{\T}{\mathbb{T}}

\let\P= \undefined
\newcommand{\P}{\mathbf{P}}

\newcommand{\E}{\mathbb{E}}

\newcommand{\F}{\mathcal{F}}

\newcommand{\al}{\alpha}
\newcommand{\be}{\beta}
\newcommand{\dl}{\delta}

\newcommand{\nb}{\nabla}

\newcommand{\Dl}{\Delta}
\newcommand{\eps}{\varepsilon}

\newcommand{\g}{\gamma}

\newcommand{\ld}{\lambda}
\newcommand{\Ld}{\Lambda}
\newcommand{\s}{\sigma}

\newcommand{\ft}{\widehat}

\newcommand{\wt}{\widetilde}
\newcommand{\cj}{\overline}

\newcommand{\dt}{\partial_t}

\newcommand{\too}{\longrightarrow}

\renewcommand{\l}{\ell}
\renewcommand{\o}{\omega}

\newcommand{\les}{\lesssim}
\newcommand{\ges}{\gtrsim}

\newcommand{\jb}[1]
{\langle #1 \rangle}

\newcommand{\ind}{\mathbf 1}

\newcommand{\N}{\mathbb{N}}

\newtheorem*{ackno}{Acknowledgements}

\numberwithin{equation}{section}
\numberwithin{theorem}{section}

\newcommand{\PP}{\mathbb{P}}

\DeclareMathOperator{\Law}{Law}

\newcommand{\dr}{\theta}
\newcommand{\Dr}{\Theta}

\newcommand{\Ha}{\mathbb{H}_a}

\usepackage{tikz}

\usetikzlibrary{shapes.misc}
\usetikzlibrary{shapes.symbols}
\usetikzlibrary{shapes.geometric}
\tikzset{
	dot/.style={circle,fill=black,draw=black,inner sep=1pt,minimum size=0.5mm},
	>=stealth,
	}
\tikzset{
	ddot/.style={circle,fill=white,draw=black,inner sep=2pt,minimum size=0.8mm},
	>=stealth,
	}

\tikzset{decision/.style={ 
        draw,
        diamond,
        aspect=1.5
    }}

\tikzset{dia2/.style
={diamond,fill=white,draw=black,inner sep=0pt,minimum size=1mm},
	>=stealth,
	}

\tikzset{dia/.style
={star,fill=black,draw=black,inner sep=0pt,minimum size=1mm},
	>=stealth,
	}

\colorlet{symbols}{black}
\colorlet{testcolor}{green!60!black}

\def\1{\mathbf{{1}}}

\usetikzlibrary{shapes.misc}
\usetikzlibrary{shapes.symbols}
\usetikzlibrary{snakes}
\usetikzlibrary{decorations}
\usetikzlibrary{decorations.markings}
\definecolor{dblue}{rgb}{0.1, 0.1, 0.9}

\tikzset{
	root/.style={circle,fill=testcolor,inner sep=0pt, minimum size=2mm},		
	dot/.style={circle,fill=black,draw=black, solid,inner sep=0pt,minimum size=0.75mm},
	bdot/.style={circle,fill=blue,draw=dblue, solid,inner sep=0pt,minimum size=0.75mm},
		}
\colorlet{symbols}{blue!90!black}

\makeatletter
\def\DeclareSymbol#1#2#3{\expandafter\gdef\csname MH@symb@#1\endcsname{\tikz[baseline=#2,scale=0.15]{#3}}%
\expandafter\gdef\csname MH@symb@#1s\endcsname{\scalebox{0.6}{\tikz[baseline=#2,scale=0.15]{#3}}}}
\def\<#1>{\csname MH@symb@#1\endcsname}
\makeatother

\DeclareSymbol{sunset}{-3}{\draw (0,0) circle [x radius = 1.5, y radius = 1]; \draw (-1.5,0) node[dot] {} -- (1.5,0) node[dot] {};}
\DeclareSymbol{tadpole}{-3}{\draw (0,0) circle [x radius = 0.75, y radius = 1]; \draw (0,-1) node[dot] {};}

\DeclareSymbol{1}{0}{\draw[white] (-.4,0) -- (.4,0); \draw (0,0)  -- (0,1.2) node[dot] {};}
\DeclareSymbol{2}{0}{\draw (-0.5,1.2) node[dot] {} -- (0,0) -- (0.5,1.2) node[dot] {};}
\DeclareSymbol{3}{0}{\draw (0,0) -- (0,1.2) node[dot] {}; \draw (-.7,1) node[dot] {} -- (0,0) -- (.7,1) node[dot] {};}
\DeclareSymbol{4}{0}{\draw (-0.36,1.2) node[dot] {} -- (0,0) -- (0.36,1.2) node[dot] {}; \draw (-1,1) node[dot] {} -- (0,0) -- (1,1) node[dot] {};}
\DeclareSymbol{30}{-3}{\draw (0,0) -- (0,-1); \draw (0,0) -- (0,1.2) node[dot] {}; \draw (-.7,1) node[dot] {} -- (0,0) -- (.7,1) node[dot] {};}
\DeclareSymbol{31}{-3}{\draw (0,0) -- (0,-1) -- (1,0) node[dot] {}; \draw (0,0) -- (0,1.2) node[dot] {}; \draw (-.7,1) node[dot] {} -- (0,0) -- (.7,1) node[dot] {};}
\DeclareSymbol{32}{-3}{\draw (0,0) -- (0,-1) -- (1,0) node[dot] {}; \draw (0,0) -- (0,-1) -- (-1,0) node[dot] {}; \draw (0,0) -- (0,1.2) node[dot] {}; \draw (-.7,1) node[dot] {} -- (0,0) -- (.7,1) node[dot] {};}
\DeclareSymbol{20}{-3}{\draw (0,0) -- (0,-1);\draw (-.7,1) node[dot] {} -- (0,0) -- (.7,1) node[dot] {};}
\DeclareSymbol{22}{-3}{\draw (0,0.3) -- (0,-1) -- (1,0) node[dot] {}; \draw (0,0.3) -- (0,-1) -- (-1,0) node[dot] {};\draw (-.7,1) node[dot] {} -- (0,0.3) -- (.7,1) node[dot] {};}
\DeclareSymbol{202}{-3}{\draw (-.6625,0) -- (0,-1) -- (.6625,0)  {}; \draw (-1.1,1) node[dot] {} -- (-.6625,0) -- (-0.365,1) node[dot] {}; \draw (0.365,1) node[dot] {} -- (.6625,0) -- (1.1,1) node[dot] {};}

\makeatletter
\def\DeclareSymbol#1#2#3{\expandafter\gdef\csname MH@symb@#1\endcsname{\tikz[baseline=#2,scale=0.15]{#3}}}
\def\<#1>{\csname MH@symb@#1\endcsname}
\makeatother

\DeclareSymbol{X}{-2.4}{\node[dot] {};}
\DeclareSymbol{1}{0}{\draw[white] (-.4,0) -- (.4,0); \draw (0,0)  -- (0,1.2) node[dot] {};}
\DeclareSymbol{2}{0}{\draw (-0.5,1.2) node[dot] {} -- (0,0) -- (0.5,1.2) node[dot] {};}

\DeclareSymbol{sunset}{-3}{\draw (0,0) circle [x radius = 1.5, y radius = 1]; \draw (-1.5,0) node[dot] {} -- (1.5,0) node[dot] {};}
\DeclareSymbol{tadpole}{-3}{\draw (0,0) circle [x radius = 0.75, y radius = 1]; \draw (0,-1) node[dot] {};}

\DeclareSymbol{T0}{-2.7}
 { \draw (0,0) node[ddot]{};}

\DeclareSymbol{T1}{0}
 {  
 \draw (0,0) node[dot]{} -- (0,4) node[ddot] {}; 
 \draw (-3,0) node[dot] {} -- (0,4)node[ddot] {} -- (3,0) node[dot] {};
\draw[line width=1pt] (0, 4)node[ddot, label=above:$r_1$]{} 
.. controls(3,6) and (5,6) ..
(8, 4)node[ddot, label=above:$r_2$]{}
(4, 12) node[label ] {$\underline{j = 1}$};
 }

\DeclareSymbol{T2}{0}
 {\draw (0,0) node[dot]{} -- (0,4) node[ddot] {}; 
 \draw (-3,0) node[dot] {} -- (0,4)node[ddot] {} -- (3,0) node[dot] {};
 \draw (0,-4) node[dot]{} -- (0,0) node[dot] {}; 
 \draw (-3,-4) node[dot] {} -- (0,0)node[dot] {} -- (3,-4) node[dot] {};
\draw[line width=1pt] (0, 4)node[ddot, label=above:$r_1$]{} 
.. controls(3,6) and (5,6) ..
(8, 4)node[ddot, label=above:$r_2$]{}
(4, 12) node[label ] {$\underline{j = 2}$};
 }

\DeclareSymbol{T3}{0}
 {\draw (0,0) node[dot]{} -- (0,4) node[ddot] {}; 
 \draw (-3,0) node[dot] {} -- (0,4)node[ddot] {} -- (3,0) node[dot] {};
 \draw (0,-4) node[dot]{} -- (0,0) node[dot] {}; 
 \draw (-3,-4) node[dot] {} -- (0,0)node[dot] {} -- (3,-4) node[dot] {};
\draw (10,0) node[dot]{} -- (10,4) node[ddot] {}; 
 \draw (7,0) node[dot] {} -- (10,4)node[ddot] {} -- (13,0) node[dot] {};
\draw[line width=1pt] (0, 4)node[ddot, label=above:$r_1$]{} 
.. controls(4,6) and (6,6) ..
(10, 4)node[ddot, label=above:$r_2$]{}
(5, 12) node[label ] {$\underline{j = 3}$};
 }

\tikzstyle{dot1} = [ draw=  gray!00, 
 rectangle, rounded corners, fill=gray!00,  inner sep=1pt, inner ysep=3pt]

\tikzstyle{dot2} = [ draw=  black, 
ellipse, fill=gray!00,  inner sep=1pt, inner ysep=3pt]

\tikzstyle{dot3} = [ draw=  gray!00, 
ellipse, fill=gray!00,  inner sep=1pt, inner ysep=3pt]

\DeclareSymbol{T4}{0}
 {\draw (0,0) node[dot1]{$\phantom{n_2^{(1)} = n^{(2)}}$} -- (0,15) node[dot1] {$\phantom{n^{(1)}}$}; 
 \draw (-13,0) node[dot1] {$n_1^{(1)}$} -- (0,15)node[ddot] {$\phantom{n^{(1)}}$} 
 -- (13,0) node[dot1] {$n_3^{(1)}$};
 \draw (0,-15) node[dot1]{$n_2^{(2)}$} -- (0,0) node[dot1] {$\phantom{n_2^{(1)} = n^{(2)}}$}; 
 \draw (-13,-15) node[dot1] {$n_1^{(2)}$} -- (0,0)node[dot1] {$n_2^{(1)} = n^{(2)}$} -- (13,-15) node[dot1] {$n_3^{(2)}$};
\draw (40,0) node[dot1]{{$n_2^{(3)}$}} -- (40,15) node[dot3]  {$\phantom{n^{(1)} = n^{(3)}}$} ; 
 \draw (27,0) node[dot1] {$n_1^{(3)}$} -- (40,15)node[dot3]  {$\phantom{n^{(1)} = n^{(3)}}$}  -- (53,0) node[dot1] {$n_3^{(3)}$};
\draw[line width=1pt] (0, 15)node[ddot, label=above:$r_1$]{$n^{(1)}$} 
.. controls(10,23) and (27,23) ..
(40, 15)node[dot2, label=above:$r_2$] {$n^{(1)} = n^{(3)}$} ;
 }

\makeatletter
\def\DeclareSymbol#1#2#3{\expandafter\gdef\csname MH@symb@#1\endcsname{\tikz[baseline=#2,scale=0.15]{#3}}}
\def\<#1>{\csname MH@symb@#1\endcsname}
\makeatother

\begin{document}
\baselineskip = 14pt

\title[Phase transition of the grand canonical $\Phi^3$ measure]
{Sharp phase transition in the grand canonical $\Phi^3$ measure at critical chemical Potential}


\author[N.~Barashkov, K.~Seong, and P. Sosoe]
{Nikolay Barashkov, Kihoon Seong, and Philippe Sosoe}

\address{Nikolay Barashkov\\
Max--Planck--Institut f\"ur Mathematik in den Naturwissenschaften\\ 
Leipzig, D-04103}

\email{nikolay.barashkov@mis.mpg.de}

\address{Kihoon Seong\\
Department of Mathematics\\
Cornell University\\ 
310 Malott Hall\\ 
Cornell University\\
Ithaca\\ New York 14853\\ 
USA }

\email{kihoonseong@cornell.edu}

\address{Philippe Sosoe\\
Department of Mathematics\\
Cornell University\\ 
310 Malott Hall\\ 
Cornell University\\
Ithaca\\ New York 14853\\ 
USA }

\email{ps934@cornell.edu}



\begin{abstract}

We study the phase transition and critical phenomenon for the grand canonical $\Phi^3$ measure in two-dimensional Euclidean quantum field theory. The study of this measure was initiated by Jaffe, Bourgain, and Carlen–Fr\"ohlich-Lebowitz, primarily in regimes far from criticality. We identify a critical chemical potential and show that the measure exhibits a sharp phase transition at this critical threshold.  At the critical threshold, the analysis is based on establishing the correlation decay of the Gaussian fluctuations in the partition function, combined with a coarse-graining argument to show divergence of the maximum of an approximating Gaussian process.




\end{abstract}

\maketitle

\tableofcontents

\setlength{\parindent}{0mm}
\setlength{\parskip}{6pt}

\section{Introduction}

\subsection{Main Results}
In this paper, we continue the study of the $\Phi^3$ measure in two-dimensional Euclidean quantum field theory, initiated by Jaffe \cite{Jaffe}, Bourgain \cite{BO95}, and Carlen–Fr\"ohlich-Lebowitz \cite{CFL}. In particular,  Bourgain \cite{BO95} and Carlen–Fröhlich–Lebowitz \cite{CFL} proposed studying the grand canonical $\Phi^3$ measure
\begin{align}
d\rho(\phi) = Z^{-1}e^{-H(\phi)}\prod_{x \in \T^2} d\phi(x),
\label{GGibs1}
\end{align}

\noi
where $Z$ is the partition function, $\T^2=(\R /  2\pi \Z)^2$, $\prod_{x\in\T^2_L}d\phi(x)$ is the (non-existent) Lebesgue measure on fields $\phi: \T^2 \to \R$,   and the grand canonical Hamiltonian is given by
\begin{align}
H(\phi)=\frac 12 \int_{\T^2} |\nb \phi|^2 dx+\frac{\s}{3} \int_{\T^2} \phi^3 dx+A\bigg(\int_{\T^2} |\phi|^2 dx \bigg)^2
\label{GGHam}
\end{align}

\noi 
for any $\s \in \R \setminus\{0\}$ and sufficiently large $A \gg 1$. Here, $\s \in \R\setminus\{0\}$ is the coupling constant measuring the strength of the cubic interaction potential\footnote{Compared to the $\Phi^4$ theory, the cubic interaction $\phi^3$ is not sign-definite and so, the sign of the coupling constant $\s$ plays no significant role. Therefore, we assume $\s \in \mathbb{R} \setminus\{0\}$.}, and the parameter $A>0$  is referred to as the chemical potential in statistical mechanics. Given $\s \in \R \setminus\{0\}$, the previous studies by Bourgain \cite{BO95} and Carlen–Fröhlich–Lebowitz \cite{CFL} focused on the regime of large chemical potential $A=A(\s) \gg 1$, far from criticality. In this paper, we identify a critical value 
\begin{align*}
A_0=A_0(\s)
\end{align*}
\noi 
of the chemical potential and show that the measure \eqref{GGibs1} exhibits a sharp phase transition at this critical threshold. For the suboptimal taming of the form $A(\int |\phi|^2 dx)^\g$, where $\g>2, A>0$, see Remark~\ref{REM:tam}.

The main difficulty in studying the Gibbs measure \eqref{GGibs1} arises from the fact that the Hamiltonian  $H$ \eqref{GGHam} with $A=0$ is unbounded from below, since the cubic interaction $\phi^3$ is not sign-definite. As a result, when $A=0$, the minimal energy satisfies $\inf H(\phi)=-\infty$ for any $\s \in \R\setminus\{0\}$. Consequently, 
$e^{-H(\phi)}$  is not integrable with respect to the Lebesgue measure $\prod_{x}d\phi(x)$ in the absence of the taming term $A(\int |\phi|^2)^2$.

In order to overcome this issue, as discussed in the works of Lebowitz-Rose-Speer \cite{LRS}, and McKean-Vaninsky \cite{McKVa}, ensembles with Hamiltonians unbounded from below are necessarily considered in a microcanonical form with respect to the particle number
\begin{align}
d\rho(\phi)=Z^{-1}e^{-\frac{\s}{3} \int_{\T^2} :\, \phi^3: \, dx } \ind_{\{\int_{\T^2} :\,\phi^2: \, dx\, \leq K\}}  d \mu(\phi),
\label{JJaffe1}
\end{align}
\noi
where $K>0$, $\mu$ is the free field, and $:\phi^3:$ and $:\phi^2:$ stand for Wick renormalizations. Note that the $\Phi^3$ measure \eqref{JJaffe1} can be constructed for any $K>0$ and any $\s \in \R\setminus\{0\}$, and thus does not exhibit a phase transition. While the $\Phi^3$ measure \eqref{JJaffe1}, as studied by Jaffe \cite{Jaffe} and also explained by Brydges–Slade \cite{BS}, is of some physical interest, as theories with cubic fields in 2$d$ have been proposed to describe the critical Potts model and percolation \cite{PL, WJ}, it does not arise as the invariant measure of any dynamics possessing a Gibbsian structure. See Remark \ref{REM:not}. In contrast, the grand canonical $\Phi^3$ measure \eqref{GGibs1} generates corresponding dynamical $\Phi^3$ models that preserve the measure.

Before introducing our main result, we emphasize that among focusing interactions (i.e.~ with Hamiltonians unbounded from below),  the cubic interaction $\s \phi^3$ is the only one that admits a meaningful formulation in two-dimensional Euclidean quantum field theory. When the cubic interaction is replaced by a higher-order interaction $\s\phi^k$, where $k \ge 5$ is odd with $\s\in \R \setminus\{0\}$, or $k \ge 4$ is even with $\s<0$, the corresponding $\Phi^k$ measure on $\T^2$ cannot be constructed, even under proper microcanonical considerations as in \eqref{JJaffe1}, or grand canonical formulations as in \eqref{GGibs1}, proved by Brydges and Slade \cite{BS};  see also \cite{OSeoT}. The failure to construct the measure for higher-order focusing interactions isolates the cubic case $\s\phi^3$ as the only remaining model amenable to rigourous study, at least in the present framework of constructive field theory.

This paper aims to identify the critical nature of the grand canonical $\Phi^3$ measure \eqref{GGibs1}, and to show that a phase transition occurs at the critical chemical potential $A=A_0$, where the threshold $A_0$ will be specified below \eqref{criA}. 
In contrast, the $\Phi^3$ measure \eqref{JJaffe1}, which is microcanonical in the particle number, does not exhibit such critical behavior, as the measure can be constructed for any $K>0$ and any $\s \in \R\setminus\{0\}$. Motivated by the above discussion, we now state our main result.  
\begin{theorem}\label{THM:1}
For any $\s \in \R \setminus \{ 0\}$, there exists  a critical chemical potential $A_0=A_0(\s)>0$
\begin{align}
A_0=\frac{\s^2}{8} \| Q^*\|_{L^2(\R^2)}^{-2},
\label{CA0}
\end{align}
\noi 
as in \eqref{criA}, where $Q^*$ is the unique solution to the elliptic equation in \eqref{GNS4}, such that the following phase transition holds:
\begin{itemize}
\item[(i)] \textup{(}Supercritical case\textup{)} For any $A< A_0$, we have 
\begin{align}
Z_A=\E_{\mu}\bigg[ e^{-\frac{\s}{3} \int_{\T^2} :\, \phi^3 :\, dx  -A\big( \int_{\T^2} :\, \phi^2: \, dx   \big)^2  }   \bigg]=\infty,
\label{super}
\end{align}
\noi
where $\mu$ denotes the massive Gaussian free field with covariance $(1-\Dl)^{-1}$. Therefore, the grand canonical $\Phi^3$ measure cannot be defined as a probability measure.

\medskip 

\item[(ii)]

\textup{(}subcritical case\textup{)} 
For any $A> A_0$, we have 
\begin{align}
Z_A=\E_{\mu}\bigg[ e^{-\frac{\s}{3} \int_{\T^2} :\, \phi^3 :\, dx  -A\big( \int_{\T^2} :\, \phi^2: \, dx   \big)^2  }   \bigg]<\infty. 
\label{sub}
\end{align}
\noi
Thus the grand canonical $\Phi^3$ measure is a well-defined:
\begin{align*}
d\rho(\phi)=Z_A^{-1}e^{-\frac{\s}{3} \int_{\T^2} :\, \phi^3 :\, dx  -A\big( \int_{\T^2} :\, \phi^2: \, dx   \big)^2  } d\mu(\phi).
\end{align*}

\medskip 

\item[(iii)]

\textup{(}critical case\textup{)} 
Let $A=A_0$. Then, we have 
\begin{align}
Z_A=\E_{\mu} \bigg[ e^{-\frac{\s}{3} \int_{\T^2} :\, \phi^3 :\, dx  -A\big( \int_{\T^2} :\, \phi^2: \, dx   \big)^2  }  \bigg]=\infty.
\end{align}
\noi
Therefore, the grand canonical $\Phi^3$ measure cannot be defined as a probability measure.

\end{itemize}

\end{theorem}

Our main theorem establishes the phase transition precisely at the critical chemical potential $A=A_0$, thereby identifying the sharp threshold for the construction of grand canonical $\Phi^3$  measure. Hence, Theorem~\ref{THM:1} fills the gap in the previous studies by Bourgain \cite{BO95} and Carlen–Fröhlich–Lebowitz \cite{CFL}, which focused on the regime of large chemical potential $A=A(\s) \gg 1$,   far from criticality.  In addition, Theorem~\ref{THM:1} addresses several questions posed by Lebowitz–Rose–Speer \cite{LRS} concerning the critical behavior and normalizability of focusing Gibbs measures at the critical threshold. See Remark~\ref{REM:LRS} for further details.


The main focus of Theorem \ref{THM:1} is the critical case $A=A_0$,  where the critical chemical potential $A_0$ is associated with the family of minimizers
\begin{align}
\mathcal{M}=\{q Q( q^{\frac 12}(\cdot-x_0) )\}_{q>0, x_0\in \R^2},
\label{solman}
\end{align}

\noi 
known as the soliton manifold, for the grand canonical Hamiltonian \eqref{GGHam} on $\R^2$. Here, the soliton profile $Q$ is radial, and decays exponentially at infinity. The structure of the soliton manifold \eqref{solman} plays a crucial role in the proof of all results in Theorem~\ref{THM:1}. For the relation between  $Q^*$ in \eqref{CA0} and the profile $Q$ in the soliton manifold, see Subsection \ref{SUBSEC:SM} and Remark \ref{REM:Rel}.

Before turning to the critical case in detail, we first look into the supercritical case  $A<A_0$ and the subcritical case $A>A_0$. In the grand canonical Hamiltonian \eqref{GGHam}, there is a competition between the cubic interaction $\frac \s3 \int \phi^3 dx$, which drives the ground state energy towards $-\infty$, and the taming by the $L^2$-norm $A\big(\int \phi^2 dx\big)^2$, acting to counterbalance the focusing nature. In the supercritical case $A<A_0$, the taming term is insufficient to control the cubic interaction $\s  \int \phi^3 dx$, leading to the divergence of the minimal energy $\inf_{\phi \in H^1} H(\phi)=\lim_{q\to \infty}H(Q_{q,x_0})=-\infty$, where $Q_{q,x_0}=qQ(q^{\frac 12}(\cdot-x_0) )$. Since the typical configuration of the $\Phi^3$  measure concentrates along the soliton manifold \eqref{solman} in the limit $q \to \infty$,  this leads to the asymptotic behavior
\begin{align*}
\log Z_A\approx -\inf_{\phi \in H^1} H(\phi)=\infty. 
\end{align*}
\noi
On the other hand, in the subcritical case $A>A_0$, the taming effect $A(\int \phi^2 dx)^2$ is sufficiently strong to control the cubic interaction $\frac \s3 \int \phi^3 dx$. Under this condition, the grand canonical Hamiltonian recovers its coercive structure, that is, $H(Q_{q,x_0})>  0$ for all $x_0\in \T^2$ and $q>0$. This coercivity ensures the normalizability $Z_A<\infty$ of the grand canonical $\Phi^3$ measure.

The most interesting case is the critical case $A=A_0$. At the critical chemical potential, the grand canonical Hamiltonian \eqref{GGHam} over $\R^2$ attains zero minimal energy along the soliton manifold, that is,  $H_{\R^2}=0$ on $\mathcal{M}$. 
Hence, in proving $Z_A=\infty$, the behavior of the partition function $\log Z_A$ at criticality is governed by the fluctuation term
\begin{align*}
\log Z_{A} = -\underbrace{\inf_{\phi \in H^1} H(\phi)}_{=0}+\,\textup{Gaussian fluctuations}.
\end{align*}
\noi
We analyze the spatial maximum of the Gaussian fluctuation term in order to show that $\log Z_A=\infty$. We study its correlation structure and carry out a coarse-graining argument to find that the maximum of the fluctuation field diverges.

More precisely, in the critical case  $A=A_0$, we divide the proof into the following five steps:


%

\begin{itemize}
\item[(i)] (\textbf{Dominant fluctuations}): In Proposition \ref{PROP:fluc}, we isolate the dominant (Gaussian) fluctuation $\Phi_{q}(x_0)$, leading to the divergence, which arises from the Cameron–Martin shift around the soliton manifold  $Q_{q,x_0}=qQ(q^{\frac 12}(\cdot-x_0) )$
\begin{align}
\log Z_{A} \approx -H(Q_{q,x_0})+ \Phi_q(x_0)\approx \Phi_q(x_0), 
\label{PF1}
\end{align}

\noi 
where we used the fact that at criticality $A=A_0$, the grand canonical Hamiltonian $H$ \eqref{GGHam} on $\T^2$ satisfies $H(Q_{q,x_0})\approx H_{\R^2}(Q_{q,x_0})=0$ as $ q \to \infty$.

\medskip 

\item[(ii)] (\textbf{Correlation decay}): In Proposition \ref{PROP:corr} we study the correlation structure of the Gaussian process $\Phi_q(x_0)$, showing that its correlation decays 
\begin{align*}
\textup{corr}(\Phi_{q,N}(x_0), \Phi_{q,N}(x_1) ) \les \frac{1}{\big( 1+q^{\frac 12}  \textup{dist}(x_0-x_1, 2\pi \Z^2) \big)^M }
\end{align*}

\noi
for any $M \ge 1$, where $x_0,x_1 \in \T^2$. From this, we identify an appropriate correlation length $\textup{dist}(x_0-x_1, 2\pi \Z^2) \big) \ges q^{-\frac 12} (\log q)^{\frac 12-\eps} $ that ensures sufficiently fast spatial decorrelation as $q\to \infty$.

\medskip

\item[(iii)] (\textbf{Coarse graining}): Based on the correlation length $\sim q^{-\frac 12} (\log q)^{\frac 12-\eps}$, we partition the torus $\T^2=(\R /  2\pi \Z)^2$ into a regular grid of squares of side $\sim q^{-\frac 12} (\log q)^{\frac 12-\eps}$. By denoting $\Ld_q$ as the collection of center points of these squares, we obtain a family of discretized Gaussian fields $\{ \Phi_{q}(x_j) \}_{j\in \Ld_q}$, indexed by the center points. In Proposition \ref{PROP:disapp}, we show that under the coarse graining scale $q^{-\frac 12} (\log q)^{\frac 12-\eps}$, the discretized approximation accurately captures the essential behavior of the continuous field
\begin{align}
\E\bigg[\max_{ x \in \T^2 } \Phi_{q}(x) \bigg]
\sim \E\bigg[\max_{x_{j}  \in \Ld_q } \Phi_{q}(x_{j}) \bigg]
\label{PF2}
\end{align}

\noi
as $q \to \infty$.



\medskip 

\item[(iv)] (\textbf{Maximum of the Gaussian process}): In Proposition \ref{PROP:max}, we analyze the maximum of the discretized Gaussian process and show that
\begin{align}
\E\bigg[\max_{x_{j}  \in \Ld_q } \Phi_{q}(x_{j}) \bigg]\sim q \sqrt{\log \# \Ld_q } \sim q \sqrt{\log q} \to \infty
\label{PF3}
\end{align}

\noi 
as $q \to \infty$. Under the chosen coarse-graining scale $q^{-\frac 12} (\log q)^{\frac 12-\eps}$, the variables $\Phi_{q}(x_j)$, $j\in \Ld_q$,  obtained from sampling the field at the points of a discrete grid of the torus are weakly correlated, and thus their maximum exhibit behavior similar to that of independent variables.

\medskip

\item[(v)] (\textbf{Divergence of the partition function}): Based on \eqref{PF1}, \eqref{PF2}, and \eqref{PF3}, we choose $x_0=\arg\max\limits_{x\in \T^2} \Phi_{q}(x)$, and thus obtain
\begin{align*}
\log Z_A\geq  \lim_{q \to \infty } \max_{x_{j}  \in \Ld_q } \Phi_{q}(x_{j})  =\infty.
\end{align*}

\noi
We then conclude the proof in the critical case $A=A_0$.

\end{itemize}

Notice that in the five steps above, we need to control ultraviolet stability (the small-scale behavior), arising from the singular nature of the free field, in the sense that all estimates remain uniform with respect to the small-scale parameters. The required control of the small-scale behavior is based on the variational approach developed by Gubinelli and the first author \cite{BG}, along with subsequent works \cite{BG2, B, BGH, CGW}.

\subsection{Remarks on the main results}

\begin{remark}\rm\label{REM:tam}
In the grand canonical Hamiltonian \eqref{GGHam}, one may consider a suboptimal taming of the form $A(\int |\phi|^2 dx)^\g$, where $\g>2$.   In this case, the corresponding grand canonical $\Phi^3$ measure can be constructed for any $A>0$ and any $\s>0$. As a result, the partition function is analytic in all parameters: $A>0, \s\in \R\setminus\{0\}$, and the inverse temperature  $\be>0$.
\end{remark}


\begin{remark}\rm\label{REM:LRS}
Lebowitz–Rose–Speer stated in \cite[Remark 5.2]{LRS}, ``Nor do we know whether the standard methods of constructive quantum field theory would suffice for Hamiltonians unbounded below".  In particular, in \cite[Remark 5.3]{LRS}, they posed the question:
``Is the measure normalizable in the critical case?".  Our main result, Theorem~\ref{THM:1}, answers these questions by establishing the (non)-normalizability of focusing Gibbs measures at the critical threshold.   These directions have since received significant attention; see, for example, \cite{OST1} on the critical behavior of the focusing Gibbs measure on the one-dimensional torus. In striking contrast to that work, we find that the $\Phi^3$ measure studied here is not normalizable at the critical value of the parameter. See also \cite{OOT}, where the $\Phi^3$ measure has been studied in three dimensions, far from the critical point. Another question posed in \cite[Remark 5.3]{LRS} is: “Are physical quantities in fact analytic in $\beta$ and $N$?” [These parameters appear in the focusing Gibbs measures.] Theorem~\ref{THM:1} shows that the partition function $Z_A$ is not analytic in the chemical potential parameter, thereby resolving this question.

\end{remark}

\begin{remark}\rm \label{REM:not}
The grand canonical $\Phi^3$ measure \eqref{GGibs1}  is the invariant measure for both the parabolic and hyperbolic dynamical $\Phi^3$-models
\begin{align}
\dt u -  \Dl  u + \s :\!u^{2}\!: +A \cdot  M_w(u) u  &=  \sqrt {2}\xi   \label{SNLH} \\
\dt^2 u+\dt u -  \Dl  u + \s :\! u^{2} \!: +A \cdot M_w(u) u &=  \sqrt {2}\xi , \label{SNLW}
\end{align} 


\noi
where $M_w(u)=\int_{\T^2} : u^2: dx $ and $\xi=\xi(x,t)$ denotes the space-time white noise on $\T^2 \times \R_+$. Notice that the $\Phi^3$ measure \eqref{JJaffe1}, which is microcanonical in the particle number, is not suitable for generating Schr\"odinger\,/\,wave\,/\,heat dynamics since (i)~the renormalized cubic power $:\! \phi^3 \!:$ makes sense only  in the real-valued setting and hence is not suitable for the Schr\"odinger equation with complex-valued solution and (ii)~\eqref{SNLH} and \eqref{SNLW} do not preserve the $L^2$-norm of a solution and thus are incompatible with the Wick-ordered $L^2$-cutoff.

\end{remark}




\section{Notations and function spaces}

\subsection{Notations}
We write $ A \les B $ to denote an estimate of the form $ A \leq CB $
for some $C> 0$. Similarly, we write  $ A \sim B $ to denote $ A \les B $ and $ B \les A $ and use $ A \ll B $ when we have $A \leq c B$ for some small $c > 0$.
We may use subscripts to denote dependence on external parameters; for example,
$A\les_{p} B$ means $A\le C(p) B$, where the constant $C(p)$ depends on a parameter $p$. 
We also use  $ a+ $ (and $ a- $) to mean  $ a + \eps $ (and $ a-\eps $, respectively)
for arbitrarily small $ \eps >0 $.

Given $N \in \N$,  we denote by $\P_N$ the Dirichlet  projection (for functions on $\T^2=(\R /  2\pi \Z)^2$)   onto frequencies $\{|n|\leq N\}$:
\begin{align}
\P_N f(x) = \sum_{|n| \leq N } \ft f(n) e^{ i n \cdot x},
\label{PNF}
 \end{align}

\noi
where the Fourier coefficient is defined as follows
\begin{align*}
\ft f(n)=\frac{1}{(2\pi)^2} \int_{\T^2} f(x)e^{-in\cdot  x}dx, \qquad n \in \Z^2.
\end{align*}

\subsection{Function spaces}
\label{SUBSEC:21}

Let $s \in \R$ and $1 \leq p \leq \infty$. We define the $L^p$-based Sobolev space $W^{s, p}(\T^2)$ by 
\begin{align*}
\| f \|_{W^{s, p}(\T^2)} = \big\| \F^{-1} [\jb{n}^s \ft f(n)] \big\|_{L^p(\T^2)}.
\end{align*}

\noi
When $p = 2$, we have $H^s(\T^2) = W^{s, 2}(\T^2)$. Let $\psi:\R \to [0, 1]$ be a smooth  bump function supported on $[-\frac{8}{5}, \frac{8}{5}]$ 
and $\psi\equiv 1$ on $\big[-\frac 54, \frac 54\big]$.
For $\xi \in \R^2$, we set $\varphi_0(\xi) = \psi(|\xi|)$
and 
\begin{align*}
\varphi_{j}(\xi) = \psi\big(\tfrac{|\xi|}{2^j}\big)-\psi\big(\tfrac{|\xi|}{2^{j-1}}\big)
\end{align*}

\noi
for $j \in \N$.
Then, for $j \in \Z_{\geq 0} := \N \cup\{0\}$, 
we define  the Littlewood-Paley projector  $\pi_j$ 
as the Fourier multiplier operator with a symbol $\varphi_j$.
Note that we have 
\begin{align*}
\sum_{j = 0}^\infty \varphi_j (\xi) = 1
\end{align*}

\noi
for each $\xi \in \R^2$ and $f = \sum_{j = 0}^\infty \pi_j f$. We next recall the basic properties of the Besov spaces $B^s_{p, q}(\T^2)$
defined by the norm
\begin{equation*}
\| u \|_{B^s_{p,q}(\T^2)} = \Big\| 2^{s j} \| \pi_{j} u \|_{L^p_x(\T^2)} \Big\|_{\l^q_j(\Z_{\geq 0})}.
\end{equation*}

\noi
We denote the H\"older-Besov space by  $\mathcal{C}^s (\T^2)= B^s_{\infty,\infty}(\T^2)$. Note that the parameter $s$ measures differentiability and $p$ measures integrability. In particular, $H^s (\T^2) = B^s_{2,2}(\T^2)$
and for $s > 0$ and not an integer, $\mathcal{C}^s(\T^2_L)$ coincides with the classical H\"older spaces $C^s(\T^2)$.

\section{Critical chemical potential}

In this section, we precisely characterize the critical chemical potential
\begin{align*}
A_0=\frac{\s^2}{8} \| Q^*\|_{L^2(\R^2)}^{-2},
\end{align*}

\noi
where $\s\in \R\setminus\{0\}$ is the coupling constant in the grand canonical Hamiltonian \eqref{GGHam}, and $Q^*$ is the optimizer of the Gagliardo–Nirenberg–Sobolev inequality, to be introduced presently. The exact form of the critical chemical potential in terms of $Q^*$ plays a crucial role in the proof of Theorem~\ref{THM:1}.

\subsection{Optimal Gagliardo–Nirenberg–Sobolev inequality.}
In this subsection, we present the optimal constant for the  Gagliardo–Nirenberg–Sobolev inequality. The optimizers of the Gagliardo–Nirenberg–Sobolev (GNS) interpolation inequality with the sharp constant $C_{\textup{GNS}}$
\begin{align}
\| \phi \|_{L^3(\R^2)}^3 \le C_{\textup{GNS}} \|  \nb \phi \|_{L^2(\R^2)} \| \phi \|_{L^2(\R^2)}^2
\label{GNS}
\end{align}

\noi
play a central role in the analysis of the two-dimensional $\Phi^3$ measure.

\begin{lemma}\label{LEM:GNS}
The functional associated with the \textup{GNS} inequality \eqref{GNS} is given by
\begin{align*}
\mathcal{F}(\phi)=\frac{\|\nb \phi \|_{L^2(\R^2)} \| \phi \|_{L^2(\R^2)}^2 }{\| \phi \|_{L^3(\R^2)}^3 }
\end{align*} 

\noi
on $H^1(\R^2)$. Then, the minimum
\begin{align*}
C_{\textup{GNS}}^{-1}:=\inf_{\substack{\phi \in H^1(\R^2) \\ \phi \neq 0 }} \mathcal{F}(\phi)
\end{align*}

\noi
is attained by a positive, radial, and exponentially decaying function $Q^*\in H^1(\R^2)$, solving the following semilinear elliptic equation on $\R^2$
\begin{align}
\Dl Q^*+2 (Q^*)^2 -2 Q^*=0
\label{GNS4}
\end{align}


\noi
with the minimal $L^2$-norm (that is, the ground state). In particular, we have 
\begin{align}
C_{\textup{GNS}}=\frac 32 \| Q^*\|_{L^2(\R^2)}^{-1}.
\label{GNS2}
\end{align}

\end{lemma}

For the proof of Lemma \ref{LEM:GNS}, see the work of Weinstein \cite{Wein}.

In the following, we use the Gagliardo–Nirenberg–Sobolev (GNS) inequality \eqref{GNS} on the torus $\T^2$, rather than on $\R^2$  as originally stated. It is important to note that the GNS inequality \eqref{GNS} does not hold for general $\phi \in H^1(\T^2)$. In particular, it fails for constant functions in $H^1(\T^2)$. Below, we state the version of the GNS inequality on $\T^2$ with the same sharp constant.

\begin{lemma}\label{LEM:GNST}
For any $\eta>0$, there exists a constant $C=C(\eta)>0$ such that 
\begin{align}
\| \phi \|_{L^3(\T^2)}^3\le   \big(C_{\textup{GNS}}+\eta \big) \| \nb \phi \|_{L^2(\T^2)}  \| \phi \|_{L^2(\T^2)}^2+C(\eta)\| \phi \|_{L^2(\T^2)}^3.
\label{GGNS}
\end{align}

\noi
for any $\phi \in H^1(\T^2)$, where $C_{\textup{GNS}}$ is as given in \eqref{GNS2}. We point out that no constant $C_0>0$ exists for which the Gagliardo–Nirenberg–Sobolev inequality 
\begin{align*}
\| \phi \|_{L^3(\R^2)}^3 \le C_{\textup{GNS}} \|  \nb \phi \|_{L^2(\R^2)} \| \phi \|_{L^2(\R^2)}^2
\end{align*}

\noi
holds for all functions in $H^1(\T^2)$.
\end{lemma}

For the proof of Lemma \ref{LEM:GNST}, see \cite[Lemma 3.3]{OST1}.

\subsection{Structure of minimizers}\label{SUBSEC:SM}

In this subsection, we study the structure of the minimizers of the following grand canonical Hamiltonian on $\R^2$
\begin{align}
H_{\R^2}(\phi)=\frac 12 \int_{\R^2} |\nb \phi|^2 dx+\frac{\s}{3} \int_{\R^2} \phi^3 dx+A\bigg(\int_{\R^2} |\phi|^2 dx \bigg)^2.
\label{Ham0}
\end{align}
\noi 
In particular, compared to the cases $A>A_0$ (unique minimizer) and $A<A_0$ (no minimizer exists), when $A=A_0$  (as given in \eqref{criA}), the Hamiltonian \eqref{Ham0} admits infinitely many minimizers, forming the so-called soliton manifold: 
\begin{align}
\mathcal{M}=\{q Q( q^{\frac 12}(\cdot-x_0) )\}_{q>0, x_0\in \R^2}.
\end{align}
\noi
Here $Q$ is a minimizer of the constrained minimization problem
\begin{align*}
\inf_{ \substack{\phi \in H^1(\R^2) \\ \| \phi\|_{L^2(\R^2) }=1 } }H_0(\phi),
\end{align*}

\noi
where
\begin{align}
H_0(\phi)=\frac 12 \int_{\R^2} |\nb \phi|^2 dx+\frac{\s}{3} \int_{\R^2} \phi^3 dx.
\label{Ham1}
\end{align}

\noi
We analyze this structure in the following lemma.
\begin{lemma}\label{LEM:MIN1}
Let $\s \in \R\setminus\{0\}$ and 
\begin{align}
A_0=\Big|\inf_{\phi \in H^1 (\R^2) } \big\{ H_0(\phi): \| \phi \|_{L^2(\R^2)}=1  \big\} \Big|.
\label{A0def}
\end{align}

\begin{itemize}
\item [(i)] Let $A>A_0$. Then the grand canonical Hamiltonian $H_{\R^2}$ \eqref{Ham0} has the unique minimizer $\phi=0$ and
\begin{align*}
\inf_{\phi \in H^1(\R^2)} H_{\R^2}(\phi)=0
\end{align*}

\medskip 

\item[(ii)] Let $A<A_0$.  Then the grand canonical Hamiltonian $H_{\R^2}$ \eqref{Ham0} admits no minimizer, and 
\begin{align*}
\inf_{\phi \in H^1(\R^2)} H_{\R^2}(\phi)=-\infty.
\end{align*}

\medskip

\item[(iii)] Let $A=A_0$.  Then the grand canonical Hamiltonian $H_{\R^2}$ \eqref{Ham0} 
admits infinitely many minimizers, given by
\begin{align}
\mathcal{M}=\{q Q( q^{\frac 12}(\cdot-x_0) )\}_{q>0, x_0\in \R^2},
\label{solman1}
\end{align}

\noi
where $Q$  is a radial Schwartz function that is positive when $\s<0$ and negative when $\s>0$. Moreover,
\begin{align*}
\inf_{\phi \in H^1(\R^2)} H_{\R^2}(\phi)=H(Q_{q,x_0})=0 
\end{align*}

\noi
for every $q>0$ and $x_0 \in \R^2$, where $Q_{q,x_0}=q Q( q^{\frac 12}(\cdot-x_0) )$.

\end{itemize}

\end{lemma}

\begin{proof}
Note that 
\begin{align}
\inf_{\phi \in H^1(\R^2)}H_{\R^2}(\phi)=\inf_{q \ge 0} \inf_{\substack{\phi \in H^1(\R^2)\\ \| \phi \|_{L^2}^2=q}}   H_{\R^2}(\phi)&=\inf_{q\ge 0} \bigg\{ \inf_{\substack{\phi \in H^1(\R^2)\\ \| \phi \|_{L^2}^2=q}} H_0(\phi)+Aq^2 \bigg\} \notag \\
&=\inf_{q\ge 0} \bigg\{  q^2\inf_{\substack{\phi \in H^1(\R^2)\\ \| \phi \|_{L^2}^2=1}} H_0(\phi)+Aq^2 \bigg\},
\label{S1}
\end{align}

\noi
where $H_0$ is the Hamiltonian defined in \eqref{Ham1}. In the second line, we used the scaling transformation $\phi_q(x)=q\phi(q^\frac 12 x)$, under which
\begin{align*}
H_0(\phi_q)=q^2H_0(\phi).
\end{align*}
\noi
Thanks to \cite[Lemma 3.4]{SS},
\begin{align}
\inf_{\substack{\phi \in H^1(\R^2)\\ \| \phi \|_{L^2}^2=1}} H_0(\phi)<0.
\label{S4} 
\end{align}

\noi
By using the definition of $A_0$ in \eqref{A0def} and \eqref{S1},  
\begin{align}
\inf_{\phi \in H^1(\R^2)}H_{\R^2}(\phi)&=\inf_{q\ge 0} \bigg\{  q^2\inf_{\substack{\phi \in H^1(\R^2)\\ \| \phi \|_{L^2}^2=1}} H_0(\phi)+Aq^2 \bigg\} \notag \\
&=\inf_{q\ge 0} \big\{ q^2(A-A_0) \big\}.
\label{S2}
\end{align}

\noi
This implies that when $A>A_0$, the minimum is achieved at $q = 0$ in \eqref{S2}. This shows that $\phi=0$ is the unique minimizer and 
\begin{align*}
\inf_{\phi \in H^1(\R^2)} H_{\R^2}(\phi)=0.
\end{align*}

\noi 
When $A<A_0$, based on \eqref{S2}, there is no minimizer and 
\begin{align*}
\inf_{\phi \in H^1(\R^2)} H_{\R^2}(\phi)=-\infty.
\end{align*}

\noi
When $A=A_0$, it follows from \eqref{S2} that 
\begin{align*}
\inf_{\phi \in H^1(\R^2)} H_{\R^2}(\phi)=0.
\end{align*}

\noi 
For any $q \ge 0$ and $x_0\in\R^2$,  define $Q_{q,x_0}:=qQ(q^{\frac 12}(\cdot-x_0) )$ where $\|Q \|_{L^2}^2=1$ and
\begin{align}
H_0(Q)=\inf_{\|\phi\|_{L^2}^2=1} H_0(\phi),
\label{S3}
\end{align}

\noi
where $H_0$ is the Hamiltonian given in \eqref{Ham1}. The existence of such a function $Q$, which is radial and belongs to the Schwartz class $\mathcal{S}(\R^2)$, follows from \cite[Lemma 3.5]{SS}. Since $\|Q_{q,x_0} \|_{L^2(\R^2)}^2=q$,
\begin{align}
H_{\R^2}(Q_{q,x_0})&=\frac{q^2}{2} \int_{\R^2} |\nb Q|^2 dx+\frac{q^2\s }{3}\int_{\R^2}Q^3 dx+Aq^2 \notag \\
&=q^2 \inf_{\substack{\phi \in H^1(\R^2)\\ \| \phi \|_{L^2}^2=1}} H_0(\phi)+Aq^2 \notag \\
&=q^2(A-A_0)=0,
\label{JS2}
\end{align}

\noi
for every $q>0$ and $x_0 \in \R^2$, where we used \eqref{S3}, \eqref{A0def}, \eqref{S4}, and $A=A_0$. This shows that $\{Q_{q,x_0}\}_{q \ge 0,x_0\in \R^2}$ forms an infinite family of minimizers.

\end{proof}

\begin{remark}\rm\label{REM:Rel}
The relationship between the optimizer $Q^*$ of the Gagliardo–Nirenberg–Sobolev inequality in Lemma~\ref{LEM:GNS} and the profile $Q$ in the soliton manifold \eqref{solman1} is given by scaling
\begin{align*}
Q^*=aQ(b(\cdot-c))
\end{align*}

\noi
for some $a,b\in \R\setminus\{0\}$ and $c\in \R^2$. This follows from the observation that the two Euler–Lagrange equations differ only by constant coefficients. By the uniqueness of solutions to the corresponding elliptic equation (up to rescaling and translation), this establishes the relation between $Q$ and $Q^*$ described above.
\end{remark}

\subsection{Optimal threshold}

In the previous subsection, we explained how the structure of the minimizers depends on the critical value $A_0$, defined in \eqref{A0def}. In the proof of Theorem~\ref{THM:1}, we use the exact expression for the critical chemical potential $A_0$ in terms of $Q^*$, the optimizer of the Gagliardo–Nirenberg–Sobolev inequality given in Lemma~\ref{LEM:GNS}.

Using the GNS inequality \eqref{GNS}, we have 
\begin{align}
H_{\R^2}(\phi)&=\frac 12 \int_{\R^2} |\nb \phi|^2 dx+\frac \s3 \int_{\R^2} \phi^3 dx+A\bigg(\int_{\R^2} \phi^2  dx \bigg)^2  \notag \\
&\ge \frac 12 \| \nb \phi \|_{L^2(\R^2)}^2 -|\s| \cdot \frac {C_{\textup{GNS}} }{3} \| \nb \phi \|_{L^2(\R^2)} \| \phi \|_{L^2(\R^2)}^2 +A \| \phi \|_{L^2(\R^2)}^4.
\label{S5}
\end{align}

\noi
Applying Young's inequality and recalling the sharp constant $C_{\textup{GNS}}=\frac 32 \| Q^*\|_{L^2(\R^2)}^{-1}$ in \eqref{GNS2}
\begin{align}
|\s|\cdot \frac {C_{\textup{GNS}} }{3} \| \nb \phi \|_{L^2(\R^2)} \| \phi \|_{L^2(\R^2)}^2 &\le \| \nb \phi \|_{L^2(\R^2)} \bigg( \frac {\s^2}{4\| Q^*\|_{L^2(\R^2)}^2} \bigg)^\frac 12 \| \phi \|_{L^2(\R^2)}^2  \notag \\
&\le \frac 12 \| \nb \phi \|_{L^2(\R^2)}^2+\frac {\s^2}{8 \| Q^*\|_{L^2(\R^2)}^2 } \| \phi \|_{L^2(\R^2)}^4.
\label{S6}
\end{align}

\noi
It follows from \eqref{S5} and \eqref{S6} that 
\begin{align}
H_{\R^2}(\phi) \ge \bigg(A-\frac{\s^2}{8 \| Q^*\|_{L^2(\R^2)}^2}  \bigg)\| \phi \|_{L^2(\R^2)}^4
\label{S7}
\end{align}

\noi
for every $\phi \in H^1(\R^2)$. 

In the following proposition, we show that the critical chemical potential $A_0$ in \eqref{A0def} is given explicitly by $A_0=\frac{\s^2}{8}\| Q^*\|_{L^2(\R^2)}^{-2}$. As a consequence of \eqref{S7}, when $A \ge A_0$, the Hamiltonian  is positive $H_{\R^2} \ge 0$.

\begin{proposition}
Let $A_0$ be the critical chemical potential defined in \eqref{A0def}. Then, we can express
\begin{align}
A_0=\frac{\s^2}{8} \| Q^*\|_{L^2(\R^2)}^{-2},
\label{criA}
\end{align}

\noi
where $\s \in \R\setminus\{0\}$, and $Q^*$ is the optimizer of the Gagliardo–Nirenberg–Sobolev inequality stated in Lemma~\ref{LEM:GNS}.

\end{proposition}

\begin{proof}
By applying the GNS inequality \eqref{GNS} under the unit mass constraint $\| \phi \|_{L^2 (\R^2) }^2=1$,
\begin{align*}
\int_{\R^2} \phi^3 dx \le C_{\textup{GNS}} \| \nb \phi \|_{L^2(\R^2)}.
\end{align*}

\noi
This implies that the Hamiltonian $H_0$ in \eqref{Ham1} satisfies 
\begin{align*}
H_0(\phi) \ge \frac 12 \| \nb \phi \|_{L^2(\R^2)}^2-\s \cdot \frac {G_{\textup{GNS}}}{3} \| \nb \phi \|_{L^2(\R^2)}
\end{align*}

\noi
under $\| \phi\|_{L^2(\R^2)}=1$. Define $\al:=\| \nb \phi \|_{L^2(\R^2)}$. Then,
\begin{align*}
H_0(\phi) \ge \Psi(\al):=\frac 12 \al^2 -\s \cdot \frac {G_{\textup{GNS}}}{3}  \al. 
\end{align*}

\noi
This quadratic $\Psi(\al)$ is minimized at $\al^*=\frac \s3 C_{\textup{GNS}}$. This implies $\Psi(\al^*)=-\frac {\s^2}{18}C_{\text{GNS} }^2 $. 
By plugging in the precise value of the optimal constant from  $C_{\textup{GNS}}=\frac 32 \| Q^{*}\|_{L^2(\R^2)}^{-1}$ \eqref{GNS2}, we obtain
\begin{align}
H_0(\phi) \ge \Psi(\al^*)=-\frac {\s^2}{18} C_{\textup{GNS}}^2=-\frac{\s^2}{18}\cdot \frac{9}{4}\cdot  \|Q^* \|_{L^2(\R^2)}^{-2}=-\frac{\s^2}{8} \| Q^*\|_{L^2(\R^2)}^{-2}
\label{GNS3}
\end{align}

\noi
under the unit mass constraint $\| \phi \|_{L^2 (\R^2) }^2=1$. That is,
\begin{align*}
\inf_{\substack{ \phi\in H^1(\R^2) \\ \|\phi \|_{L^2(\R^2)}=1} }H_0(\phi)\ge -\frac{\s^2}{8} \| Q^*\|_{L^2(\R^2)}^{-2}.
\end{align*}

We now show that the inequality \eqref{GNS3} is actually an equality. Consider the profile
\begin{align*}
\phi_\ld(x):=\frac{\ld Q^*(\ld x)}{ \|Q^* \|_{L^2(\R^2)}},
\end{align*}

\noi
where $Q^*$ is the GNS optimizer \eqref{GNS4}, and $\ld \in \R \setminus\{0\}$. Then, we have $\| \phi_\ld \|_{L^2(\R^2)}=1$. Note that
\begin{align*}
\int_{\R^2} \phi_\ld^3 dx &=\frac{\ld }{ \| Q^*\|_{L^2(\R^2)}^3 } \int_{\R^2} (Q^*)^3 dx\\
\int_{\R^2} |\nb \phi_\ld|^2 dx&= \frac{\ld^2}{  \| Q^*\|_{L^2(\R^2)}^2 } \| \nb Q^*\|_{L^2(\R^2)}^2.
\end{align*}

\noi
This implies that
\begin{align}
H_0(\phi_\ld)= \frac 12 \cdot \frac{\ld^2 }{ \|Q^* \|_{L^2(\R^2)}^2} \| \nb Q^* \|_{L^2(\R^2)}^2+\frac \s3 \cdot \frac{\ld }{  \|Q^*\|_{L^2(\R^2)}^3 } \int_{\R^2} (Q^*)^3 dx.
\label{GNS6}
\end{align}

\noi
Since $Q^*$ is the optimizer for the GNS inequality \eqref{GNS} from Lemma \ref{LEM:GNS}, 
\begin{align}
\| Q^*\|_{L^3(\R^2)}^3&= \bigg(\frac 32 \| Q^*\|_{L^2(\R^2)}^{-1} \bigg)\| \nb Q^*\|_{L^2(\R^2)}\| Q^* \|_{L^2(\R^2)}^2 \notag \\
&=\frac 32 \|Q^* \|_{L^2(\R^2)} \| \nb Q^* \|_{L^2(\R^2)}.
\label{GNS5}
\end{align}

\noi
By plugging \eqref{GNS5} into \eqref{GNS6}
\begin{align*}
H_0(\phi_\ld)&= \frac 12 \cdot \frac{\ld^2 }{ \|Q^* \|_{L^2(\R^2)}^2} \| \nb Q^* \|_{L^2(\R^2)}^2+\frac \s2 \cdot \frac{\ld }{  \|Q^*\|_{L^2(\R^2)}^2 } \|  \nb Q^*\|_{L^2(\R^2) }\\
&=\frac{   \| \nb Q^* \|_{L^2(\R^2)}^2       }{\|Q^*\|_{L^2(\R^2)}^2} \cdot \bigg(\frac 12 \ld^2 +\frac \s2 \cdot \frac{\ld} { \| \nb Q^* \|_{L^2(\R^2)}} \bigg).
\end{align*}

\noi
By optimizing the quadratic part in $\ld \in \R \setminus\{0\}$, we choose 
\begin{align*}
\ld^*=-\frac{\s}{2}\| \nb Q^* \|_{L^2(\R^2)}^{-1}.
\end{align*}

\noi 
This implies that 
\begin{align}
H_0(\phi_{\ld^*})=\frac{   \| \nb Q^* \|_{L^2(\R^2)}^2       }{\|Q^*\|_{L^2(\R^2)}^2} \cdot \bigg(-\frac {\s^2}8 \cdot   \frac {1}{ \| \nb Q^* \|_{L^2(\R^2)}^2} \bigg)=-\frac{\s^2}{8} \| Q^*\|_{L^2(\R^2)}^{-2}
\label{GNS7}
\end{align}

\noi
By combining \eqref{GNS3} with \eqref{GNS7}, we obtain
\begin{align*}
A_0=\Big|\inf_{\phi \in H^1 (\R^2) } \big\{ H_0(\phi): \| \phi \|_{L^2(\R^2)}=1  \big\} \Big|=\frac{\s^2}{8} \| Q^*\|_{L^2(\R^2)}^{-2}.
\end{align*}

\end{proof}

\section{Bou\'e-Dupuis variational formalism for the Gibbs measure}
\label{SUBSEC:var}

In this subsection, we introduce a framework for analyzing expectations of certain random fields under the free field.  Let $(\Omega, \mathcal{F},\mathbb{P})$ be a probability space on which is defined a space-time white noise $\xi$ on $\T^2 \times \R_{+}$. Let $W(t)$ be the cylindrical Wiener process on $L^2(\T^2)$ with respect to the underlying probability measure $\PP$. That is,
\begin{align*}
W(t) =\sum _{n \in \Z^2 } B_n(t)e^{in \cdot x},
\end{align*}

\noi
where $\{ B_n \}_{n \in \Z^2 }$ is defined by  $B_n(t)=\jb{\xi, \ind_{[0,t]} \cdot e^{in \cdot x}  }_{\T^2\times \R} $. Here, $\jb{\cdot, \cdot}_{\T^2 \times \R}$ denotes the duality pairing on $\T^2 \times \R$ and $\xi$ is a space-time white noise on $\T^2 \times \R_{+}$. Then, we see that $\{B_n \}_{n \in \Z^2 }$ is a family of mutually independent complex-valued Brownian motions conditioned\footnote{In particular, $B_0$ is  a standard real-valued Brownian motion.}  to have $B_{-n}=\cj{B_n }$, $n \in \Z^2$. We  then define a centered Gaussian process $ Y(t) $ by 
\begin{align}
Y(t)  &= \jb{\nabla}^{-1} W(t)
\label{Rol}
\end{align}

\noi
where $\jb{\nb}=(1-\Dl)^\frac 12$. Then, we have $\Law ( Y(1)) = \mu$.  By setting  $Y_{N}(t) = \P_N {Y}(t)$, we have   $\Law ({Y}_{N}(1)) = (\P_N)_\#\mu$, where $\P_N $  is the Fourier projector onto the frequencies  $\{|n| \le N\}$; see \eqref{PNF}. For later use we also set $Q_{q,x_0,N}:=\P_N Q_{q,x_0}$ for a soliton $Q_{q,x_0}$. We define the second and third Wick powers of $Y_{N}$ as follows 
\begin{align}
:\!Y_N(t)^2\!:
&=Y_N^2(t) - \<tadpole>_{N}(t), \label{cher18}\\
:\! Y_N(t)^3 \!:
&=Y_N(t)^3 -3\<tadpole>_{N}(t) Y_N(t) \label{cher19}.
\end{align}

\noi
Here,
\begin{align*}
\<tadpole>_{N}(t):=\E \Big[| Y_N(t) |^2\Big] =\sum_{\substack{n \in \Z^2 \\ |n|\le N}} \frac{t}{\jb{n}^2}   \sim t \log N,
\end{align*}

\noi
where $\jb{n} = (1 +  |n|^2)^\frac{1}{2}$.

Next, let $\mathbb{H}_a $ denote the space of drifts, which are the progressively measurable processes\footnote{With respect to the filtration $\mathcal{F}_t=\sigma(B_n(s), n \in \mathbb{Z}^2, 0\le s\le t)$.} belonging to $L^2([0,1]; L^2(\T^2))$, $\PP$-almost surely. We are now ready to state  the  Bou\'e-Dupuis variational formula \cite{BD, Ust,HW}. The version we cite here comes from \cite{HW}. See also Theorem 2 in \cite{BG} and Theorem 7 in~\cite{Ust}, where the same conclusion is obtained under stronger assumptions.

\begin{lemma}\label{LEM:var3}
Let $Y(t)=\jb{\nabla}^{-1} W(t)$ be as in \eqref{Rol}.
Fix $N \in \N$.
Suppose that  $F:C^\infty(\T^2) \to \R$
is measurable such that
$\E\big[|e^{-F(\P_N Y(1))} \big] < \infty$ and $\E\big[F_-(\P_N Y(1))\big] < \infty$, where $F_-=\max\{0,-F\}$.
Then, we have
\begin{align*}
\log \E\Big[e^{-F(\P_N Y(1))}\Big]
&= \sup_{\dr \in \mathbb{H}_{a} }
\E\bigg[ -F(\P_N Y(1) + \P_N \Dr(1)) - \frac{1}{2} \int_0^1 \| \dr(t) \|_{L^2 }^2 dt \bigg]\\
&= \sup_{\dr \in \mathbb{H}_{a} }
\E\bigg[ -F(\P_N Y(1) + \P_N \Dr(1)) - \frac{1}{2} \int_0^1 \|  \dot{\Dr} (t) \|_{H^1 }^2 dt \bigg],
\end{align*}

\noi
where  $\Dr$ is  defined by 
\begin{align}
 \Dr(t) = \int_0^t \jb{\nabla}^{-1} \dr (t') dt'
\label{DR}
\end{align}

\noi
and the expectation $\E = \E_\PP$
is an 
expectation with respect to the underlying probability measure~$\PP$.

\end{lemma}

In the following, we set $Y_{N}=\P_N {Y}(1)$ and $\Dr_{N}=\P_N \Dr(1)$ for $N\in \N\cup \{\infty\} $, where $\P_\infty=\Id$ is understood to be the identity operator. Before we move to the next subsection, we state a lemma on the pathwise regularity bounds of  $Y(1)$ and $\Dr(1)$.

\begin{lemma}  \label{LEM:Dr}
	
\textup{(i)} 
Let $\eps > 0$. Then, given any finite $p \ge 1$, $0\leq t\leq 1$ 
\begin{align}
\begin{split}
\E 
\Big[ & \|Y_N(t)\|_{\mathcal{C}^{-\eps} }^p
+ \|:\!Y_N(t)^2\!:\|_{\mathcal{C}^{-\eps} }^p
+ 
\big\| :\! Y_N(t)^3 \!:  \big\|_{  \mathcal{C}^{-\eps} }^p
\Big]
\leq C_{\eps, p} <\infty,
\end{split}
\label{EE0}
\end{align}

\noi
uniformly in $N \in \N$. In addition, for $k=2,3$, we have 
\begin{align}
\E \Bigg[ \bigg| \int_{\T^2} :\! Y_N^k \!: dx\bigg|^p \Bigg] \le C_{k,p}<\infty,
\label{E0}
\end{align}

\noi
uniformly in $N \in \N$.
    
\smallskip
	
\noi
\textup{(ii)} For any $\dr \in \Ha$ and $0 \le t \le 1$, we have
\begin{align}
\| \Dr(t) \|_{H^{1}}^2 \leq \int_0^1 \| \dr(s) \|_{L^2}^2ds.
\label{CM}
\end{align}
\end{lemma}

Part (i) of Lemma \ref{LEM:Dr} follows
from a standard computation and thus we omit details.
As for Part (ii),  the estimate \eqref{CM}
follows  from Minkowski's and Cauchy-Schwarz' inequalities
\begin{align*}
\| \Dr(t) \|_{H^{1}} \le \int_0^1 \| \dr(s) \|_{L^2}ds \le \bigg( \int_0^1 \| \dr(s) \|_{L^2}^2 ds\bigg)^{\frac 12} .
\end{align*}

\section{Supercritical case}

In this subsection, we discuss the failure of constructing the grand canonical $\Phi^3$ measure when $A<A_0$, stated in Theorem \ref{THM:1}. In other words, we prove that when $A<A_0$, where $A_0=A_0(\s)$ is given by \eqref{criA}, 
\begin{align*}
Z_A=\E_{\mu}\bigg[ e^{-\frac{\s}{3} \int_{\T^2} :\, \phi^3 :\, dx  -A\big( \int_{\T^2} :\, \phi^2: \, dx   \big)^2  }   \bigg]=\infty. 
\end{align*}

\noi
for any $\s \in \R\setminus\{0\}$. The idea is that typical configurations under the measure concentrate around the soliton manifold, that is, the family of minimizers \eqref{solman1}
\begin{align*}
\big\{ qQ(q^\frac 12(\cdot-x_0))  \big\}_{q>0, x_0 \in \R^2 }
\end{align*}

\noi
and thus
\begin{align*}
\inf_{\phi \in H^1(\R^2)} H_{\R^2}(\phi)=-\infty,
\end{align*}

\noi
as established in Lemma~\ref{LEM:MIN1}.

\begin{proof}[Proof of Theorem \ref{THM:1} \textup{(i)}]

Define the partition function $Z_{A,N}$ with the ultraviolet cutoff $\P_N$ \eqref{PNF}
\begin{align*}
Z_{A,N}:=\E_{\mu}\bigg[ e^{-\frac{\s}{3} \int_{\T^2} :\, \phi_N^3 :\, dx  -A\big( \int_{\T^2} :\, \phi_N^2: \, dx   \big)^2  }   \bigg],
\end{align*}

\noi
where $\phi_N=\P_N\phi$. By the Bou\'e-Dupuis variational formula in Lemma \ref{LEM:var3}, we have  
\begin{align*}
\log Z_{A,N}=\sup_{\dr \in \Ha }\E\Bigg[&-\s \int_{\T^2} Y_N \Dr_N^2 dx-\s \int_{\T^2 } :\! Y_N^2 \!: \Dr_N dx -\frac \s3 \int_{\T^2} \Dr_N^3 dx \\
&-A\bigg( \int_{\T^2} :\! Y_N^2 \!: +2Y_N \Dr_N +\Dr_N^2 dx \bigg)^2-\frac 12 \int_0^1 \|\dot{\Dr}(t) \|^2_{H^1  } dt \Bigg],
\end{align*}

\noi
where $\dot{\Dr}(t)=\frac{d}{dt}\Dr(t)=\jb{\nb}^{-1}\dr(t)$ from \eqref{DR}. Choosing $\Dr(t)=t Q_{q,x_0}$, where $Q_{q,x_0}=qQ(q^{\frac 12}(\cdot-x_0))$,  $q> 0$, $x_0 \in \T^2$, and using the fact that $Y_N$ and $:\! Y_N^2 \!:$ are centered, we have 
\begin{align}
\log Z_{A,N} &\ge -\frac \s3 \int_{\T^2} Q_{q,x_0,N}^3 dx-A\bigg(\int_{\T^2} Q_{q,x_0,N}^2 dx \bigg)^2-\frac 12 \int_{\T^2} |\nb Q_{q,x_0} |^2 dx  \notag \\
&\hphantom{X}+\mathcal{E}(Y_N, Q_{q,x_0,N})-\frac{1}{2} \int_{\T^2} Q_{q,x_0}^2 dx, 
\label{RRR0}
\end{align}

\noi
where $\mathcal{E}(Y_N, Q_{q,x_0,N})$  plays the role of an error term
\begin{align}
\mathcal{E}(Y_N, Q_{q,x_0,N})&=\E\Bigg[ -A\bigg( \int_{\T^2} :\! Y_N^2 \!: dx \bigg)^2 -A\bigg( \int_{\T^2} 2Y_N Q_{q,x_0,N} dx  \bigg)^2   \notag \\
&\hphantom{XXX}-2A\bigg(\int_{\T^2} :\! Y_N^2 \!: dx\bigg) \bigg( \int_{\T^2 } Q_{q,x_0,N}^2 dx \bigg) \notag \\
&\hphantom{XXX}-2A\bigg(\int_{\T^2} :\! Y_N^2 \!: dx\bigg) \bigg( \int_{\T^2 } 2Y_N Q_{q,x_0,N} dx \bigg)  \notag \\
&\hphantom{XXX}-2A\bigg(\int_{\T^2} 2Y_N Q_{q,x_0,N} dx\bigg) \bigg( \int_{\T^2 }  Q_{q,x_0,N}^2 dx\bigg)\Bigg]  \notag \\
&=\I_1+\I_2+\I_3+\I_4+\I_5.
\label{RR0}
\end{align}

\noi
Since $Y_N$ and $:\!Y_N^2\!:$ are centered, $Q_{q,x_0}$ is deterministic, and $\E[:\!Y_N^2(x)\!:Y_N(y)]=0$, we have 
\begin{align}
\I_3=\I_4=\I_5=0.
\label{RR1}
\end{align}

\noi
Thanks to the definition of the free field $Y_N$ \eqref{Rol}, 
\begin{align}
\I_2=\E\Bigg[  \bigg|    \int_{\T^2} Y_N Q_{q,x_0,N} dx  \bigg|^2   \Bigg]=\| Q_{q,x_0,N} \|_{H^{-1}}^2\sim \|Q \|_{L^2(\R^2)}^2
\label{RR2}
\end{align}

\noi
as $q\to \infty$. Combining \eqref{RR0}, \eqref{RR1}, and \eqref{RR2} yields 
\begin{align}
\mathcal{E}(Y_N, Q_{q,x_0,N})=O(1).
\label{RRR2}
\end{align}

\noi
Since $Q_{q,x_0}=qQ(q^{\frac 12}(\cdot-x_0) )$ is a highly localized profile with exponential decay as $q \to \infty$ ($Q$ is a Schwartz function), we have
\begin{align}
\frac{1}{2} \int_{\T^2} Q_{q,x_0}^2 dx=\frac{1}{2} \int_{\R^2} Q_{q,x_0}^2 dx +e^{-cq}\sim q+e^{-cq}
\label{JS1} 
\end{align}

\noi 
for some $c>0$, as $q \to \infty$.

\noi
Based on Lemmas \ref{LEM:A1} and \ref{LEM:A2}, we choose $N=N(q)=q^{\frac 52+\eps}$, and from \eqref{RRR0} and \eqref{JS1}, we obtain
\begin{align}
\log Z_{A,N(q)}
& \ges -H_{\R^2}(Q_{q,x_0})-e^{-cq}+O(q^{-\eps})+\mathcal{E}(Y_N, Q_{q,x_0,N})-q,
\label{R0}
\end{align}

\noi
where in the last step, we used the fact that the ground state $Q$ has an exponential decay on $\R^2$ and so $H(Q_{q,x_0})\sim H_{\R^2}(Q_{q,x_0})+e^{-cq}$ as $q \to \infty$.
Here, $H_{\R^2}$ means the grand canonical Hamiltonian on $\R^2$. Under the condition $A<A_0$, it follows from \eqref{JS2} that 
\begin{align}
H_{\R^2}(Q_{q,x_0})=q^2H_{\R^2}(Q)=(A-A_0)q^2=-cq^2
\label{R1}
\end{align}

\noi
for some $c>0$. Combining \eqref{R0}, \eqref{R1}, and \eqref{RRR2} yields 
\begin{align*}
\log Z_{A,N(q)} \ges  cq^2- e^{-cq}+O(q^{-\eps})+O(1)-q
\end{align*}

\noi
for some small $\eps>0$. By taking the limit $q\to \infty$, we obtain 
\begin{align*}
Z_A=\E_{\mu}\bigg[ e^{-\frac{\s}{3} \int_{\T^2} :\, \phi^3 :\, dx  -A\big( \int_{\T^2} :\, \phi^2: \, dx   \big)^2  }   \bigg]=\infty. 
\end{align*}

\noi
This completes the proof of Theorem \ref{THM:1} (i).

\end{proof}

Note that 
\begin{align*}
Q_{q,x_0,N}(x)=\varphi_N*Q_{q,x_0}(x)=\int_{\T^2} Q_{q,x_0}(x-y)N^2 \varphi(Ny) dy.
\end{align*}

\noi
In order to get  $Q_{q,x_0,N}(x) \approx Q_{q,x_0}(x)$, the ultraviolet (small-scale) cutoff 
$N$ should depend on the scaling parameter $q$ in such a way that $N \gg q$.  In the following lemmas, we derive the exact relation between $N$ and $q$ through quantitative estimates.

\begin{lemma}\label{LEM:A1}
We obtain the following quantitative estimate
\begin{align*}
\Bigg| \bigg( \int_{\T^2} Q_{q,x_0,N}^2 dx \bigg)^2-\bigg( \int_{\T^2} Q_{q,x_0}^2 dx \bigg)^2  \Bigg| \les N^{-1} q^{\frac 32},
\end{align*}

\noi
uniformly in $x_0 \in \T^2$. In particular, under the condition $N=q^{\frac 32+\eps}$, we have 
\begin{align*}
\bigg( \int_{\T^2} Q_{q,x_0,N}^2 dx \bigg)^2 \to \bigg( \int_{\T^2} Q_{q,x_0}^2 dx \bigg)^2
\end{align*}

\noi
as $q\to \infty$.

\end{lemma}

\begin{proof}
Note that 
\begin{align}
\bigg| \int_{\T^2} (Q_{q,x_0,N}^2 -Q_{q,x_0}^2 ) dx \bigg| \les \|Q_{q,x_0,N}-Q_{q,x_0}  \|_{L^2} \cdot \big(\| Q_{q,x_0,N}  \|_{L^2}+\| Q_{q,x_0} \|_{L^2}\big).
\label{Q1}
\end{align}

\noi
Thanks to the exponential decay of the ground state $Q$ on $\R^2$, we have that as $q \to \infty$ 
\begin{align}
\|Q_{q,x_0} \|_{L^2(\T^2)} \sim \| q Q(q^{\frac 12}(\cdot-x_0) )  \|_{L^2(\R^2)}=q^{\frac 12} \|Q \|_{L^2(\R^2)}.
\label{Q2}
\end{align}

\noi
Since $\| Q_{q,x_0,N} \|_{L^2(\T^2)} \les \| Q_{q,x_0} \|_{L^2(\T^2)} $, we have $\| Q_{q,x_0,N} \|_{L^2(\T^2)}  \les q^{\frac 12}$, uniformly in $N \ge 1$.

\noi
Recall the standard mollifier estimate
\begin{align*}
\|f*\phi_N-f \|_{L^2 } \les N^{-1} \| \nb f \|_{L^2}.
\end{align*}

\noi
This implies that 
\begin{align}
\|Q_{q,x_0,N} -Q_{q,x_0} \| \les N^{-1} \| \nb Q_{q,x_0} \|_{L^2} \sim N^{-1}q.
\label{Q3}
\end{align}

\noi
Combining \eqref{Q1}, \eqref{Q2}, and \eqref{Q3} yields that
\begin{align}
\bigg| \int_{\T^2} (Q_{q,x_0,N}^2 -Q_{q,x_0}^2 ) dx \bigg| \les N^{-1}q \cdot q^{\frac 12}=N^{-1}q^{\frac 32}.
\label{Q4}
\end{align}

\noi
This shows that \eqref{Q4} vanishes as $q\to \infty$ if $N \gg q^{\frac 32}$.

\end{proof}

\begin{lemma}\label{LEM:A2}
We obtain the following quantitative estimate
\begin{align*}
\bigg| \int_{\T^2} Q^3_{q,x_0,N} dx -\int_{\T^2} Q^{3}_{q,x_0} dx  \bigg| \les N^{-1} q^{\frac{5}{2}},
\end{align*}

\noi
uniformly in $x_0 \in \T^2$. In particular, under the condition $N=q^{\frac {5}{2}+\eps}$, we have 
\begin{align*}
\int_{\T^2} Q_{q,x_0,N}^3 dx \too \int_{\T^2} Q_{q,x_0}^3 dx
\end{align*}

\noi
as $q \to \infty$.

\end{lemma}

\begin{proof}

By H\"older's inequality, 
\begin{align}
\bigg| \int_{\T^2} Q^3_{q,x_0,N} dx -\int_{\T^2} Q^{3}_{q,x_0} dx  \bigg| \les \| Q_{q,x_0,N}-Q_{q,x_0} \|_{L^3} \cdot \big( \| Q_{q,x_0,N} \|_{L^3}^2+\| Q_{q,x_0} \|_{L^3}^2  \big).
\label{QQ1}
\end{align}

\noi 
Using the exponential decay of the ground state $Q$ on $\R^2$, we have that as $q \to \infty$ 
\begin{align}
\|Q_{q,x_0} \|_{L^3(\T^2)} \sim \| q Q(q^{\frac 12}(\cdot-x_0) )  \|_{L^3(\R^2)}=q^{\frac 23} \|Q \|_{L^3(\R^2)}.
\label{QQ2}
\end{align}

\noi 
Since $\| Q_{q,x_0,N} \|_{L^3} \les \| Q_{q,x_0} \|_{L^3} $, we have $\| Q_{q,x_0,N} \|_{L^3}  \les q^{\frac 23}$, uniformly in $N \ge 1$.

Recall the standard mollifier estimate as in \eqref{Q3}, we have 
\begin{align}
\| Q_{q,x_0,N}-Q_{q,x_0} \|_{L^3}  \les N^{-1} \| \nb Q_{q,x_0} \|_{L^3} \sim N^{-1} q^{\frac 76}.
\label{QQ3}
\end{align}

\noi
It follows from \eqref{QQ1}, \eqref{QQ2}, and \eqref{QQ3} that 
\begin{align}
\bigg| \int_{\T^2} Q^3_{q,x_0,N} dx -\int_{\T^2} Q^{3}_{q,x_0} dx  \bigg| \les N^{-1 } \cdot q^{\frac 76} \cdot q^{\frac 43}=N^{-1}q^{\frac 52}.
\label{QQ4}
\end{align}

\noi
This shows that \eqref{QQ4} vanishes as $q\to \infty$ if $N \gg q^{\frac 52}$.

\end{proof}

\section{Subcritical case}

In this subsection, we prove Theorem \ref{THM:1} (ii). In other words, when $A>A_0$, where $A_0=A_0(\s)$ is given by \eqref{criA}, we have  
\begin{align}
Z_A=\E_{\mu}\bigg[ e^{-\frac{\s}{3} \int_{\T^2} :\, \phi^3 :\, dx  -A\big( \int_{\T^2} :\, \phi^2: \, dx   \big)^2  }   \bigg]<\infty 
\end{align}

\noi
for any $\s  \in \R \setminus\{0\}$. Notice that from Lemma \ref{LEM:MIN1} (i), when $A>A_0$, we have  $H_{\R^2} \ge 0$. That is, the grand canonical Hamiltonian recovers its coercive structure.

\begin{proof}[Proof of Theorem \ref{THM:1} \textup{(ii)}]
Recall the partition function $Z_{A,N}$ with ultraviolet cutoff $\P_N$ \eqref{PNF}
\begin{align*}
Z_{A,N}:=\E_{\mu}\bigg[ e^{-\frac{\s}{3} \int_{\T^2} :\, \phi_N^3 :\, dx  -A\big( \int_{\T^2} :\, \phi_N^2: \, dx   \big)^2  }   \bigg],
\end{align*}

\noi
where $\phi_N=\P_N\phi$. By using the Bou\'e-Dupuis variational formula in Lemma \ref{LEM:var3} and Lemma \ref{LEM:Dr} (ii), we write  
\begin{align*}
\log Z_{A,N}\le \sup_{\dr \in \Ha }\E\Bigg[&-\s \int_{\T^2} Y_N \Dr_N^2 dx-\s \int_{\T^2 } :\! Y_N^2 \!: \Dr_N dx -\frac \s3 \int_{\T^2} \Dr_N^3 dx \\
&-A\bigg( \int_{\T^2} :\! Y_N^2 \!: +2Y_N \Dr_N +\Dr_N^2 dx \bigg)^2-\frac 12 \| \Dr_N \|_{H^1}^2 \Bigg].
\end{align*}

\noi
By expanding the taming term, we obtain
\begin{align}
\log Z_{A,N} \le   \sup_{ \dr \in \Ha } \E \Big[ -H(\Dr_N)-\Psi_1(Y_N,\Dr_N)-\Psi_2(Y_N,\Dr_N) -\frac 12 \| \Dr_N \|_{L^2}^2 \Big],
\label{JS3}
\end{align}

\noi
where 
\begin{align*}
\Psi_1(Y_N,\Dr_N)&=\s \int_{\T^2} :\! Y_N^2 \!: \Dr_N dx + \s \int_{\T^2} Y_N \Dr_N^2 dx\\
\Psi_2(Y_N, \Dr_N)&=A\bigg( \int_{\T^2} :\! Y_N^2 \!: dx \bigg)^2+ 4A\bigg(   \int_{\T^2} Y_N  \Dr_N dx \bigg )^2+4A \bigg(\int_{\T^2} :\! Y_N^2 \!: dx\bigg) \bigg( \int_{\T^2} Y_N \Dr _N dx\bigg)\\
&\hphantom{X}+2A\bigg(\int_{\T^2} :\! Y_N^2 \!: dx\bigg)\bigg( \int_{\T^2} \Dr_N^2  dx \bigg) +4A \bigg( \int_{\T^2} Y_N \Dr_N dx\bigg) \bigg( \int_{\T^2} \Dr_N^2  dx \bigg).
\end{align*}

\noi
By applying Lemmas \ref{LEM:Dr} and \ref{LEM:error2}, we obtain bounds on the error terms
$\Psi_1$ and $\Psi_2$ 
\begin{align}
\E \big|\Psi_1(Y_N,\Dr_N)  \big| &\le  \eps \E \| \Dr_N \|_{H^1}^2 +\eps \E \| \Dr_N \|_{L^2}^4+C_\eps  \label{JS4}\\
 \E \big|\Psi_2(Y_N,\Dr_N)  \big| &\le  \eps \E \| \Dr_N \|_{H^1}^2 +\eps  \E \| \Dr_N \|_{L^2}^4+C_\eps, \label{JS5}
\end{align}

\noi
for arbitrarily small $\eps>0$, where $C_\eps \gg 1$ arises from estimating higher moments of the stochastic objects $(Y_N, \; :\! Y_N^2 \!:, \; :\! Y_N^3 \!:)$ using Lemma \ref{LEM:Dr}. By combining \eqref{JS3}, \eqref{JS4}, and \eqref{JS5}, we obtain
\begin{align}
\log Z_{A,N} &\le \sup_{ \dr \in \Ha } \E \bigg[ -H (\Dr_N) +\eps \| \nb \Dr_N \|_{L^2}^2 +\eps \| \Dr_N \|_{L^2}^4  -\Big(\frac 12-\eps \Big) \| \Dr_N \|_{L^2}^2  \bigg]+ C_\eps  \notag \\
&\le  - \E \Big[  \inf_{\dr \in \Ha } H^* (\Dr_N)   \Big]+ C_\eps,
\label{JS6}
\end{align}

\noi
where
\begin{align*}
H^{*}(\phi)=\Big(\frac 12 -\eps \Big)\int_{\T^2} | \nb \phi|^2 dx+\frac {\s}3 \int_{\T^2} \phi^3 dx+(A-\eps)\bigg( \int_{\T^2} \phi^2 dx \bigg)^2.
\end{align*}

\noi 
By using the GNS inequality \eqref{GGNS},
\begin{align}
H^*(\phi)&=\Big(\frac 12-\eps \Big) \int_{\T^2} |\nb \phi|^2 dx+\frac \s3 \int_{\T^2} \phi^3 dx+(A-\eps)\bigg(\int_{\T^2} \phi^2  dx \bigg)^2  \notag \\
&\ge \Big(\frac 12-\eps \Big) \| \nb \phi \|_{L^2(\T^2)}^2 - |\s| \cdot \frac {C_{\textup{GNS} }  +\eta }{3} \| \nb \phi \|_{L^2(\T^2)} \| \phi \|_{L^2(\T^2)}^2 \notag\\
&\hphantom{X}+(A-\eps) \| \phi \|_{L^2}^4-C(\eta) \| \phi \|_{L^2(\T^2)}^3.
\label{W1}
\end{align}

\noi
By applying Young's inequality 
\begin{align*}
ab \le \frac{\g^2}{2}a^2+\frac{1}{2\g^2}b^2
\end{align*}

\noi
for any $a,b>0$ with $\g=\sqrt{1-2\eps}$, together with the sharp constant $C_{\textup{GNS}}=\frac 32 \| Q^*\|_{L^2(\R^2)}^{-1}$ in \eqref{GNS2}, we obtain 
\begin{align}
&|\s|\cdot \frac {C_{\textup{GNS}} +\eta }{3} \| \nb \phi \|_{L^2(\T^2)} \| \phi \|_{L^2(\T^2)}^2 \notag \\ 
&= \| \nb \phi \|_{L^2(\T^2)} \bigg( \frac {|\s|}{2\| Q^*\|_{L^2(\R^2)} }+\frac{\eta |\s|}{3} \bigg) \| \phi \|_{L^2(\T^2)}^2  \notag \\
&\le \Big( \frac 12 -\eps \Big)\| \nb \phi \|_{L^2(\T^2)}^2+\frac 12\bigg(\frac {|\s|}{2 \| Q^*\|_{L^2(\R^2)} } +\frac{\eta |\s|}{3} \bigg)^2 ( 1-2\eps )^{-1}\| \phi \|_{L^2(\T^2)}^4 \notag \\
&=\Big( \frac 12 -\eps \Big)\| \nb \phi \|_{L^2(\T^2)}^2+\bigg(\frac{\s^2}{8  \| Q^*\|_{L^2(\R^2)}^2 }+O(\eta)\bigg)(1-2\eps)^{-1} \| \phi \|_{L^2(\T^2)}^4.
\label{W0}
\end{align}

\noi
By combining \eqref{W1} and \eqref{W0}, 
\begin{align*}
H^{*}(\phi) \ge \Bigg(A-\eps-\bigg(\frac{\s^2 }{8 \|  Q^*\|_{L^2(\R^2)}^2} +O(\eta)  \bigg)(1-2\eps)^{-1}  \Bigg)\| \phi \|_{L^2(\T^2)}^4-C(\eta) \|  \phi\|_{L^2(\T^2)}^3. 
\end{align*}

\noi
Using the subcritical condition  $A>A_0$, where $A_0$ is defined in \eqref{criA} 
\begin{align*}
A_0=\frac{\s^2}{8 \| Q^* \|_{L^2(\R^2)}^2 },
\end{align*}

\noi
and choosing $\eta, \eps$ sufficiently small, we write 
\begin{align}
H^{*}(\phi)  \ge \al \| \phi \|_{L^2(\T^2)}^4-C(\eta) \| \phi \|_{L^2(\T^2)}^3,
\label{quapol}
\end{align}

\noi
where
\begin{align*}
\al=A-\eps-\bigg(\frac{\s^2 }{8 \|  Q^*\|_{L^2(\R^2)}^2} +O(\eta)  \bigg)(1-2\eps)^{-1}>0.
\end{align*}

\noi 
Since the leading-order coefficient $\al>0$ in the quartic polynomial \eqref{quapol} is positive, we obtain 
\begin{align}
\inf_{\phi \in H^1} H^*(\phi) \ge - C>-\infty
\label{JS7}
\end{align}

\noi
for some constant $C>0$. It follows from \eqref{JS6} and \eqref{JS7} that 
\begin{align*}
\log Z_{A,N} \le - \E \bigg[  \inf_{\dr \in \Ha } H^* (\Dr_N)   \bigg]+ C_\eps \le \wt C_\eps<\infty,
\end{align*}

\noi
uniformly in $N\ge 1$, where $\wt C_\eps$ is a large constant depending on $\eps>0$.

\end{proof}

Before concluding this subsection, we present Lemma \ref{LEM:error2}, which was used to control the error terms $\Psi_1$ and $\Psi_2$ in \eqref{JS4} and \eqref{JS5}.

\begin{lemma}\label{LEM:error2}
For every $\dl>0$, we have 
\begin{align*}
\bigg| \int_{\T^2}  :\! Y_N^2 \!: \Dr_N dx \bigg|
&\le C_\dl \| :\! Y_N^2 \!: \|_{\mathcal{C}^{-\eps} }^2  
+ \dl 
\| \Dr_N \|_{H^1 }^2, 
\notag  \\
\bigg| \int_{\T^2 }  Y_N \Dr_N^2  dx \bigg|
&\le C_\dl 
\| Y_N \|_{ \mathcal{C}^{-\eps} }^{p_1} + \dl \Big(
\| \Dr_N \|_{H^1 }^2 +  \| \Dr_N \|_{L^2}^4 \Big) \\
\bigg|   \int_{\T^2} Y_N  \Dr_N dx \bigg |^2  &\le C_\dl \| Y_N \|_{ \mathcal{C}^{-\eps} }^{p_2} + \dl \Big(\| \Dr_N \|_{H^1 }^2 +  \| \Dr_N \|_{L^2}^4 \Big)\\
\bigg| \int_{\T^2} Y_N \Dr_N dx \cdot \int_{\T^2} \Dr_N^2  dx \bigg| & \le   C_\dl \| Y_N \|_{ \mathcal{C}^{-\eps} }^{p_3} + \dl \Big(\| \Dr_N \|_{H^1 }^2 +  \| \Dr_N \|_{L^2}^4 \Big)
\end{align*}

\noi
for some large exponents $p_1,p_2,p_3>1$, where $C_\dl$ is a constant that blows up as $\dl \to 0$, that is, $C_\dl \to \infty$ as $\dl \to 0$.

\end{lemma}

\begin{proof}
The estimates follow from Besov space duality, embedding, and Young's inequality. For details, see \cite[Lemma 3.5]{OSeoT}.
\end{proof}

\section{Critical case}

We now consider the (non-)construction of the grand canonical $\Phi^3$ measure at the critical threshold $A=A_0$,
\begin{align*}
A_0=\frac{\s^2}{8 }\| Q^*\|_{L^2(\R^2)}^{-2},
\end{align*}

\noi
as given in \eqref{criA}, where a phase transition occurs. In the following, we fix the coupling constant $\s=1$,  as it plays no essential role.

\subsection{Characterizing dominant Gaussian fluctuations}

In the critical case $A=A_0$, the structure of the family of minimizers (i.e.~the soliton manifold)
\begin{align}
\{Q_{q,x_0}\}_{q>0, x_0\in \R^2},
\end{align}

\noi 
where $Q_{q,x_0}=q Q( q^{\frac 12}(\cdot-x_0) )$, plays a crucial role in proving the non-construction of the $\Phi^3$ measure  for the grand canonical Hamiltonian \eqref{GGHam}.  Notice that  the minimal energy along the soliton manifold is zero $\inf_{\phi \in H^1}H(\phi)=0$, that is,  $H(Q_{q,x_0}) = 0 $ for every $ q>0 $ and $ x_0 \in \mathbb{R}^2$. This implies that compared to (i) the supercritical case $A<A_0$, where the minimal energy is $-\infty$, and (ii) the subcritical case $A>A_0$, where $H(Q_{q,x_0})>0$ for every $q>0$ and $x_0 \in \T^2$,  the behavior of the partition function $\log Z_A$ at criticality is governed by the fluctuation term
\begin{align*}
\log Z_{A} \approx -\underbrace{\inf_{\phi \in H^1} H(\phi)}_{=0}+\,\textup{fluctuations}.
\end{align*}

\noi
In the following proposition, we give a candidate for the fluctuation part  $\Phi_{q,N}(x_0)$, which later leads to divergence.

\begin{proposition}\label{PROP:fluc}
Let $A=A_0$, where $A_0$ is the critical chemical potential as defined in \eqref{criA}. Then, by choosing $N=N(q)=q^{\frac 52+\eps}$, we obtain 
\begin{align*}
\log Z_{A,N(q)} \ge \E\bigg[ \max_{x \in \T^2} \Phi_{q,N(q)}(x) \bigg] -q -e^{-cq },
\end{align*}

\noi
as $q \to \infty$, where 
\begin{align}
\Phi_{q,N}(x_0)=-\int_{\T^2} \Delta Q_{q, x_0  } Y_{N(q)}(\tfrac 12) dx
\label{GA0}
\end{align}

\noi
is a Gaussian process over $x_0 \in \T^2$. Here, the Gaussian process $Y_{N(q)}(t)=\P_{N(q)}Y(t)$ in \eqref{Rol} is evaluated at $t=\frac 12$.


\end{proposition}


\begin{proof}
In the Boué–Dupuis formula (Lemma \ref{LEM:var3}), we choose a drift $\dr^*(t)$
\begin{align}
\dr^*(t)=2 \jb{\nb} \cdot qQ(q^{\frac 12}(\cdot-x_0)) \ind_{ \{ \tfrac 12 \le t\le 1 \} }(t),
\label{dr}
\end{align}

\noi
where $x_0 \in \T^2$ is the (random) point at which $\Phi_{q,N}(x)$ attains its maximum
\begin{align}
x_0=\arg\max\limits_{x\in \T^2} \Phi_{q,N}(x).
\label{choicex0}
\end{align}

\noi
Here, $\Phi_{q,N}$ is the Gaussian process in \eqref{GA0}.  Then, by the definition of $\Dr=\Dr(1)$ in \eqref{DR}, we have 
\begin{align}
\Dr=\Dr(1)=\int_0^1 \jb{\nb}^{-1}\dr^*(t)dt=qQ(q^\frac 12(\cdot-x_0))=Q_{q,x_0}
\label{GA2} 
\end{align}

\noi
Notice that since the (random) point $x_0$ is chosen to maximize $\Phi_{q,N}$, where $Y_N(t)$ is evaluated at $t=\frac 12$ (see \eqref{GA0}), and the cutoff $ \ind_{ \{ \tfrac 12 \le t\le 1 \} }$ is inserted, the drift  $\dr$ in \eqref{dr} is an admissible choice that satisfies the measurability condition with respect to the filtration $\mathcal{F}_t$, that is, $\dr^* \in \Ha $. Regarding the measurability issue associated with the choice of $\dr^*(t)$, see Remark \ref{REM:dr}.

By plugging \eqref{GA2} and \eqref{dr} into the Boué–Dupuis formula (Lemma \ref{LEM:var3}),
\begin{align}
\log Z_{A,N } & \ge \E\bigg[  - \int_{\T^2} :\!Y_N^2\!: Q_{q,x_0,N } dx- \int_{\T^2}Y_N Q_{q,x_0,N}^2 dx -\frac  13\int_{\T^2} Q_{q,x_0,N}^3 dx \notag \\
&\hphantom{XXX}- A\bigg( \int_{\T^2} :\! Y_N^2 \!: dx+ 2\int_{\T^2} Y_N Q_{q,x_0,N}dx+\int_{\T^2} Q_{q,x_0,N}^2 dx  \bigg)^2 \notag \\
&\hphantom{XXX}-\frac 12 \int_{\T^2} |\nb Q_{q,x_0}|^2 dx -\frac 12 \int_{\T^2} |Q_{q,x_0}|^2 dx  \bigg],
\label{C1}
\end{align}

\noi
where $Q_{q,x_0,N}=\P_N Q_{q,x_0}$. 
Note that 
\begin{align*}
- \int_{\T^2} Y_N Q_{q,x_0,N}^2  dx-4A \bigg( \int_{\T^2} Q_{q,x_0,N }^2  dx \bigg)\bigg( \int_{\T^2}  Q_{q,x_0}Y_N dx\bigg)
\end{align*}

\noi
is the main contribution to the Gaussian process $\Phi_{q,N}(x_0)$ in \eqref{GA0}, while the remaining terms in \eqref{C1} act as error terms. We expand the taming part as follows 
\begin{align}
&\bigg( \int_{\T^2} :\! Y_N^2 \!: dx+ 2\int_{\T^2} Y_N Q_{q,x_0,N}dx+\int_{\T^2} Q_{q,x_0,N}^2 dx  \bigg)^2 \notag \\
&=\bigg(   \int_{\T^2} :\! Y_N^2 \!: dx \bigg )^2+4 \bigg(   \int_{\T^2} Y_N Q_{q,x_0,N}dx \bigg )^2 +\bigg( \int_{\T^2} Q_{q,x_0,N}^2 dx  \bigg)^2 \notag \\
&\hphantom{X}+4 \bigg(\int_{\T^2} :\! Y_N^2 \!: dx\bigg) \bigg( \int_{\T^2} Y_N Q_{q,x_0,N}dx\bigg)+2\bigg(\int_{\T^2} :\! Y_N^2 \!: dx\bigg)\bigg( \int_{\T^2} Q_{q,x_0,N}^2  dx \bigg) \notag \\
&\hphantom{X}+4 \bigg( \int_{\T^2} Y_N Q_{q,x_0,N }dx\bigg) \bigg( \int_{\T^2} Q_{q,x_0,N}^2  dx \bigg).
\label{C2}
\end{align}

\noi
Recall that in the critical case $A=A_0$, $\{Q_{q,x_0}\}_{q>0, x_0\in \R^2}$, where $Q_{q,x_0}=qQ(q^{\frac 12}(\cdot-x_0) )$, forms a set of minimizers for the following grand canonical Hamiltonian
\begin{align*}
H_{\R^2}(\phi)=\frac 12 \int_{\R^2} |\nb \phi|^2 dx+\frac 13 \int_{\R^2} \phi^3 dx+A\bigg( \int_{\R^2} \phi^2 dx \bigg)^2
\end{align*}

\noi
with minimal energy $H_{\R^2}(Q_{q,x_0})=0$ for all $q>0$ and $x_0 \in \R^2$. Therefore, each $Q_{q,x_0}$ satisfies the Euler–Lagrange equation at the critical chemical potential $A=A_0$
\begin{align}
-\Dl Q_{q,x_0}+  Q_{q,x_0}^2+4A\bigg( \int_{\R^2} Q_{q,x_0}^2 dx \bigg) Q_{q,x_0}=0.
\label{EU0}
\end{align}

\noi
Since $Q_{q,x_0}=qQ(q^{\frac 12}(\cdot-x_0) )$ is a highly localized profile with exponential decay as $q \to \infty$ ($Q$ is a Schwartz function), we have $\| Q_{q,x_0}\|_{L^2(\R^2)}^2=\| Q_{q,x_0}\|_{L^2(\T^2)}^2+g(q)$ where $|g(q)|\leq exp(-cq)$ for some $c>0$. This implies that 
\begin{align}
\eqref{EU0}=-\Dl Q_{q,x_0}+  Q_{q,x_0}^2+4A\bigg( \int_{\T^2} Q_{q,x_0}^2 dx \bigg) Q_{q,x_0}+4Ag(q) Q_{q,x_0}=0.
\label{EU1}
\end{align}

\noi
By applying the ultraviolet cutoff $\P_N$ (that is, frequency projection onto $\{|n| \le N\}$), 
\begin{align}
0&=-\Dl Q_{q,x_0,N}+ \P_N \big( Q_{q,x_0}^2 \big)+4A\bigg( \int_{\T^2} Q_{q,x_0}^2  dx \bigg) Q_{q,x_0,N}+4Ag(q) Q_{q,x_0,N}  \notag \\
&=-\Dl Q_{q,x_0,N}+ Q_{q,x_0,N}^2 +4A\bigg( \int_{\T^2} Q_{q,x_0,N}^2 dx \bigg)Q_{q,x_0,N}    \notag \\
&\hphantom{X}+\textup{com}(\P_N \big( Q_{q,x_0}^2 \big),Q_{q,x_0,N}^2 )+4A\big(g(q) + O(N^{-c}) \big)Q_{q,x_0,N},
\label{EU2}
\end{align}

\noi
where we used $\| Q_{q,x_0}\|_{L^2(\T^2)}^2=\| Q_{q,x_0,N}\|_{L^2(\T^2)}^2+O(N^{-c}) $ for some $c>0$. Here, the commutator $\textup{com}(\P_N \big( Q_{q,x_0}^2 \big),Q_{q,x_0,N}^2 )= \P_N \big( Q_{q,x_0}^2 \big)- (\P_N Q_{q,x_0})^2 $ satisfies
\begin{align*}
\|\textup{com}(\P_N \big( Q_{q,x_0}^2 \big),Q_{q,x_0,N}^2 ) \|_{L^2}  \les N^{-\eps } \| Q_{q,x_0}\|_{H^\eps}^2.
\end{align*}

\noi 
Using the Euler–Lagrange equation \eqref{EU2} with the projection $\P_N$,  we write 
\begin{align}
- \int_{\T^2} Q_{q,x_0,N}^2 Y_N dx-4A \bigg( \int_{\T^2} Q_{q,x_0,N }^2  dx \bigg)\bigg( \int_{\T^2}  Q_{q,x_0}Y_N dx\bigg)  =\Phi_{q,N}(x_0,1)+\mathcal{E}_1(Y_N, Q_{q,x_0,N}),
\label{C3}
\end{align}

\noi
where 
\begin{align}
\Phi_{q,N}(x_0, 1)=-\int_{\T^2} \Delta Q_{q, x_0 } Y_N(1) dx
\label{GAA0}
\end{align}

\noi
is a Gaussian process over $x_0 \in \T^2$ (with $Y_N=Y_{N}(1)$ in \eqref{Rol} evaluated at $t=1$).   Here, $\mathcal{E}_1(Y_N, Q_{q,x_0,N})$ is an error term
\begin{align*}
\mathcal{E}_1(Y_N, Q_{q,x_0,N})=\int_{\T^2 } \textup{com}(\P_N \big( Q_{q,x_0}^2 \big),Q_{q,x_0,N}^2 ) Y_N dx+4A\big(e^{-cq} + O(N^{-c}) \big) \int_{\T^2} Y_N Q_{q,x_0} dx.
\end{align*}

Based on Lemmas \ref{LEM:A1} and \ref{LEM:A2}, we choose $N=N(q)=q^{\frac 52+\eps}$ to get  
\begin{align}
-\frac 13 \int_{\T^2} Q_{q,x_0,N}^3 dx-A\bigg(\int_{\T^2} Q_{q,x_0,N}^2 dx \bigg)^2-\frac 12 \int_{\T^2} |\nb Q_{q,x_0} |^2 dx=-H(Q_{q,x_0})+O(q^{-\eps}). 
\label{C4}
\end{align}

\noi
as $q \to \infty$. Combining \eqref{C1}, \eqref{C2}, \eqref{C3}, and \eqref{C4} yields 
\begin{align}
\log Z_{A,N}&\ge \E\Big[ -H(Q_{q,x_0}) +\Phi_{q,N}(x_0,1)+ \mathcal{E}(Y_N,Q_{q,x_0,N} ) \Big]+O(q^{-\eps}),
\label{J1}
\end{align}

\noi
where $\mathcal{E}(Y_N, Q_{q,x_0,N})$  plays the role of an error term
\begin{align*}
\mathcal{E}(Y_N, Q_{q,x_0,N})&=-\int_{\T^2 } :\! Y_N^2 \!:   Q_{q,x_0} dx  -A\bigg(   \int_{\T^2} :\! Y_N^2 \!: dx \bigg )^2\\
&\hphantom{X} -4A\bigg(\int_{\T^2}  Y_N Q_{q,x_0} dx \bigg)^2-  4 A \bigg(\int_{\T^2} :\! Y_N^2 \!: dx\bigg) \bigg( \int_{\T^2} Y_N Q_{q,x_0}dx\bigg)\\
&\hphantom{X}-2A\bigg(\int_{\T^2} :\! Y_N^2 \!: dx\bigg)\bigg( \int_{\T^2} Q_{q,x_0}^2  dx \bigg)+4A\big(e^{-cq} + O(N^{-c}) \big) \int_{\T^2} Y_N Q_{q,x_0} dx\\
&\hphantom{X}-\frac 12 \int_{\T^2} |Q_{q,x_0}|^2 dx+\mathcal{E}_1(Y_N, Q_{q,x_0,N}).
\end{align*}

\noi
Thanks to \eqref{E0}, Lemma \ref{LEM:error1} and \eqref{EE0}, we obtain the following error estimate
\begin{align}
\mathbb{E}|\mathcal{E}(Y_N, Q_{q,x_0,N}) | \les q,
\label{J2}
\end{align}

\noi
uniformly in $N \ge 1$. Since $Q_{q,x_0}=qQ(q^{\frac 12}(\cdot-x_0) )$ is a highly localized profile with exponential decay as $q \to \infty$, we have   
\begin{align}
H(Q_{q,x_0})&=H_{\R^2}(Q_{q,x_0})+\mathcal{O}(e^{-cq})=\mathcal{O}(e^{-cq})
\label{J3}
\end{align}

\noi
for some $c>0$, where we used the fact that $H_{\R^2}(Q_{q,x_0})=0$ for all $q>0$ and $x_0 \in \R^2$ at the critical chemical potential $A=A_0$. Combining \eqref{J1}, \eqref{J2}, and \eqref{J3} yields that 
\begin{align}
\log Z_{A,N}&\ge \E\Big[ -H(Q_{q,x_0}) +\Phi(x_0)+ \mathcal{E}(Y_N,Q_{q,x_0,N} ) \Big]+O(q^{-\eps}) \notag \\
& \ges  -e^{-cq }+\E\big[ \Phi_{q,N}(x_0,1) \big] -q
\label{GA3}
\end{align}

\noi
as $q\to \infty$.

Recall that $x_0$ is chosen as a random point measurable w.r.t $\mathcal{F}_{\frac 12}$ in \eqref{choicex0} and $Y_N(t)$ is a martingale. Hence, 
\begin{align}
\E\big[ \Phi_{q,N}(x_0,1) \big]&=\E\Big[ \E\big[\Phi_{q,N}(x_0,1)\big| \mathcal{F}_{\frac 12}  \big] \Big] \notag \\
&=-\int_{\T^2} \E\Big[ \Dl Q_{q,x_0} \E\big[  Y_N(1)  \big| \mathcal{F}_{\frac 12}  \big] \Big] dx \notag \\
&=-\int_{\T^2} \E\big[ \Dl Q_{q,x_0} Y_N(\tfrac 12) \big] dx \notag \\
&=\E\big[\Phi_{q,N}(x_0)\big],
\label{GA4}
\end{align}

\noi
where the Gaussian processes $\Phi_{q,N}(x_0,1)$ and $\Phi_{q,N}(x_0)$ are defined in \eqref{GAA0} and \eqref{GA0}, respectively. Combining \eqref{GA3},\eqref{GA4}, and \eqref{choicex0}, we obtain
\begin{align*}
\log Z_{A,N} \ge \E\bigg[ \max_{x \in \T^2} \Phi_{q,N}(x) \bigg] -q -e^{-cq }. 
\end{align*}

\noi 
This completes the proof of Proposition \ref{PROP:fluc}.

\end{proof}

\begin{remark}\rm \label{REM:dr}
If we choose the random point $x_0$ as
\begin{align*}
x_0:=\arg\max\limits_{x\in \T^2} \Phi_{q,N}(x,1),
\end{align*}

\noi 
where $\Phi_{q,N}(x,1)$ is defined in \eqref{GAA0}, then $Q_{q,x_0}$ is not adapted to the filtration $\mathcal{F}_t$, $t<1$. As a result, the corresponding choice of $\dr^*(t)$, as defined in \eqref{dr}, is not admissible; that is, $\dr^* \notin \Ha$.
\end{remark}

Before proceeding to the next subsection, we present the lemma used in the proof of Proposition \ref{PROP:fluc}.

\begin{lemma}\label{LEM:error1}
Let $\{Q_{q,x_0} \}_{q>0, x_0 \in \T^2}$ be the soliton manifold. Then, 
\begin{align*}
\bigg| \int_{\T^2 } :\! Y_N^2 \!:   Q_{q,x_0} dx  \bigg|&\les  \| :\! Y_N^2 \!: \|_{\mathcal{C}^{-\eps}}  q^{\frac \eps2}\\
\bigg| \int_{\T^2 } Y_N   Q_{q,x_0} dx  \bigg|&\les  \|  Y_N \|_{\mathcal{C}^{-\eps}}  q^{\frac \eps2}\\
\bigg| \int_{\T^2 }  Q_{q,x_0}^2 dx  \bigg|&\sim q.
\end{align*}

\noi
uniformly in $N \ge 1$ and $x_0 \in \T^2$.

\end{lemma}

\begin{proof}
The first two estimates follow from Besov space duality, embedding, and Young's inequality. Regarding the last estimate, since $Q_{q,x_0}=qQ(q^{\frac 12}(\cdot-x_0) )$ is a highly localized profile with exponential decay as $q \to \infty$ ($Q$ is a Schwartz function), we have $\| Q_{q,x_0}\|_{L^2(\R^2)}^2=\| Q_{q,x_0}\|_{L^2(\T^2)}^2+O(e^{-cq})$ for some $c>0$, where the error term is uniform in $x_0 \in T^2$. 
\end{proof}

\subsection{Correlation decay}

In this subsection, we study the correlation decay of the Gaussian process over $x_0\in \T^2$, which arises from the dominant fluctuation term $\Phi_{q,N}(x_0)$ in Proposition \ref{PROP:fluc} 
\begin{align*}
\Phi_{q,N}(x_0)=\int_{\T^2} Q_{q,x_0,N}(x) \Dl Y_N (\tfrac 12,x) dx.
\end{align*}

\noi 
In the following proposition, we prove the strong correlation decay as $q\to \infty$.
\begin{proposition}\label{PROP:corr}
Let $\Phi_{q,N}(x_0)$ be the Gaussian process over $x_0 \in \T^2$, as defined in Proposition~\ref{PROP:fluc}. Then, by choosing $N=N(q)=q^{\frac 52+}$, as in Proposition \ref{PROP:fluc}, we obtain 
\begin{align*}
\textup{corr}(\Phi_{q,N}(x_0), \Phi_{q,N}(x_1) )&:=\frac{\E\big[\Phi_{q,N}(x_0)  \cj {\Phi_{q,N}(x_1) } \big]}{\big(\E \big[ |\Phi_{q,N}(x_0)|^2 \big]\big)^\frac 12  \big(\E \big[ |\Phi_{q,N}(x_1)|^2 \big] \big)^\frac 12 }\\
&\les_M \frac{1}{\big( 1+q^{\frac 12}  \textup{dist}(x_0-x_1, 2\pi \Z^2) \big)^M }
\end{align*}

\noi
for any $M \ge 1$, where the implicit constant depends on $M$ and  
\begin{align*}
\textup{dist}(x_0-x_1, 2\pi \Z^2) =\inf_{k \in \Z^2} | x_0-x_1- 2\pi k|.
\end{align*}

\noi
This shows strong correlation decay as $q \to \infty$ with correlation length $q^{-\frac 12}(\log q)^{\frac 12-\eps}$. Moreover, the variance is given by
\begin{align*}
\E \big[ |\Phi_{q,N}(x_0)|^2 \big] \sim q^2
\end{align*}

\noi
as $q \to \infty$.

\end{proposition}

\begin{proof}

Recall that 
\begin{align}
\Phi_{q,N}(x_0)=\int_{\T^2} Q_{q,x_0,N}(x) \Dl Y_N (\tfrac 12, x) dx= \sum_{|n| \le N} \frac{|n|^2}{\sqrt{1+|n|^2}} B_n(\tfrac 12,\o)  \int_{\T^2}  Q_{q,x_0}(x)e^{in \cdot x}dx,
\label{D00}
\end{align}

\noi
where $B_n(t,\o)$ denotes a Brownian motion.   We define an error term $\mathcal{E}_q(n)$ as follows 
\begin{align}
\int_{\T^2} qQ(q^{\frac 12} (x-x_0) ) e^{in \cdot x} dx&=\int_{\R^2} qQ(q^{1/2}(x-x_0) ) e^{in \cdot x} dx +\mathcal{E}_q(n) \notag \\
&=e^{in \cdot x_0}\ft Q(q^{-\frac 12}n) +\mathcal{E}_q(n).
\label{D0}
\end{align}

\noi
In the following, we prove
\begin{align}
|\mathcal{E}_q(n)| \les \frac{e^{-cq^{\dl}}}{\jb{n}^M}
\label{DD1}
\end{align}

\noi 
for some $\dl,c>0$ and every $M \ge 1$. Note that 
\begin{align}
\int_{\T^2} qQ(q^{\frac 12} (x-x_0) ) e^{in \cdot x} dx= e^{in \cdot x_0} \int_{  y\in q^{\frac 12}(\T^2-x_0) } Q(y) e^{i q^{-\frac 12}n \cdot y} dy. 
\label{D1}
\end{align}

\noi
Based on the definition \eqref{D0} of $\mathcal{E}_q(n)$, \eqref{D1}  implies that  
\begin{align*}
|\mathcal{E}_q(n)|=\bigg| \int_{\R^2 \setminus q^{\frac 12}(\T^2-x_0) } Q(y) e^{iq^{-\frac 12}n \cdot y } dy  \bigg| \les \frac{1}{\jb{|q^{-\frac 12} n | }^M  } e^{-cq^\dl} \les \frac{e^{-\frac c2 q^{\dl}}}{\jb{n}^M},
\end{align*}

\noi
where we used 
\begin{align*}
e^{i q^{-\frac 12} n \cdot y }=\frac{1}{ | q^{-\frac 12 }n |^2} \Dl_y \big(e^{i q^{-\frac 12} n \cdot y }\big)
\end{align*}

\noi
and $Q$ is a Schwartz function.  Combining \eqref{D00} and \eqref{D0} yields 
\begin{align}
\Phi_{q,N}(x_0)= \sum_{|n| \le N} \frac{|n|^2}{\sqrt{1+|n|^2} } B_n(\tfrac 12,\o)  \big( e^{in \cdot x_0} \ft Q(q^{-\frac 12}n ) +\mathcal{E}_q(n) \big).
\label{DD22}
\end{align}

\noi
We now study the correlation function 
\begin{align}
\E[\Phi_{q,N}(x_0) \cj{ \Phi_{q,N}(x_1) } ]&=\frac 12 \cdot \sum_{ |n| \le N}  \frac{|n|^4}{1+|n|^2 } |\ft Q(q^{-1/2} n  ) |^2 e^{in \cdot (x_0-x_1)}+O(e^{-cq^{\dl}}) \notag \\ 
&=\frac{1}{2}S_q(x_0,x_1)+O(e^{-cq^{\dl}})+\frac 12 \cdot \sum_{|n| > N} \frac{|n|^4}{1+|n|^2 } |\ft Q(q^{-1/2} n  ) |^2 e^{in \cdot (x_0-x_1)} 
\label{D2}
\end{align}

\noi
where we used the independence of $B_n$ and \eqref{DD1}.  Regarding the tail estimate in \eqref{D2}, we use the fact that $Q$ is a Schwartz function, together with the condition
$N=N(q)=q^{\frac 52+}$ from Proposition \ref{PROP:fluc}, to obtain
\begin{align}
\sum_{|n|>N} |n|^2 (q^{-\frac 12} |n|)^{-2M} \les q^{M}  \sum_{|n| \ge N} |n|^{2-2M} \les q^{M}  N^{4-2M} \les q^{-4M+10}
\label{DD2} 
\end{align}

\noi
for any $M \ge $1.

We now study the term $S_q(x_0,x_1)$ in \eqref{D2}. Recall the Poisson summation formula
\begin{align}
\sum_{n \in \Z^2} f(n ) e^{in \cdot x}=\sum_{k \in \Z^2} \ft f (x+2\pi k),
\label{D3}
\end{align}

\noi
where $\ft f(n)=\int_{\R^2} f(x) e^{-in \cdot x} dx$ is the Fourier transform of 
$f$, evaluated at the lattice points $n\in \Z^2$.  By recalling the definition \eqref{D2} of  $S_q(x_0,x_1)$ and applying the Poisson summation formula \eqref{D3}, we obtain
\begin{align}
S_q(x_0, x_1)&=\sum_{ n \in \Z^2 }  f(n) e^{in \cdot (x_0-x_1)} \notag \\
&= \sum_{k \in \Z^2} \ft f ((x_0-x_1) +2\pi k),
\label{D4}
\end{align}

\noi
where  $f(n)=\frac{|n|^4}{1+|n|^2} | \ft Q(q^{-1/2 } n)   |^2$.   Here, 
\begin{align*}
\ft f(x)=\int_{\R^2} f(\xi) e^{i\xi \cdot x} d\xi=q^2 \int_{\R^2} \frac{|\xi|^4}{1/q+|\xi|^2} |\ft Q (\xi)|^2 e^{i \xi \cdot  q^\frac 12 x } d\xi.
\end{align*}

\noi
Since $\nb_\xi (\xi \cdot q^{\frac 12} x)=q^{\frac 12} x\neq 0$, we apply the non-stationary phase method, namely, repeated integration by parts in $\xi$, with
\begin{align*}
e^{i \xi \cdot  q^\frac 12 x }=\frac{1}{|q^{\frac 12 }x |^2} \Dl_\xi (e^{i \xi \cdot  q^\frac 12 x })
\end{align*}

\noi
to obtain 
\begin{align}
|\ft f(x)| \les \frac{q^2}{(1+q^{1/2}|x| )^M }.
\label{D5}
\end{align}

\noi
for every $M \ge 1$. It follows from \eqref{D4} and \eqref{D5} that 
\begin{align}
S_q(x_0,x_1)&= \sum_{k \in \Z^2} \ft f ((x_0-x_1) +2\pi k) \notag \\
&\les \sum_{ k \in \Z^2} \frac{q^2}{(1+q^{\frac 12}|(x_0-x_1) +2\pi k| )^M } \notag \\
&\les \frac{q^2}{(1+q^{\frac 12}|(x_0-x_1) +2\pi k_0| )^M }+\sum_{ k \neq k_0} \frac{q^2}{(1+q^{\frac 12}|(x_0-x_1) +2\pi k| )^M } \notag \\
&\les  \frac{q^2}{\big( 1+q^{\frac 12}  \text{dist}(x_0-x_1, 2\pi \Z^2) \big)^M }
\label{DD5}
\end{align}

\noi
for any $M \ge 1$, where 
\begin{align*}
\text{dist}(x_0-x_1, 2\pi \Z^2) =\inf_{k \in \Z^2} | x_0-x_1- 2\pi k|=| x_0-x_1- 2\pi k_0|.
\end{align*}

\noi
Combining \eqref{D2}, \eqref{DD2}, and \eqref{DD5} yields 
\begin{align*}
\E\big[\Phi_{q,N}(x_0) \cj{ \Phi_{q,N}(x_1) } \big] &\les \frac{q^2}{\big( 1+q^{\frac 12}  \text{dist}(x_0-x_1, 2\pi \Z^2) \big)^M }+O(e^{-cq^\dl})+q^{-4M+10}\\
&\les \frac{q^2}{\big( 1+q^{\frac 12}  \text{dist}(x_0-x_1, 2\pi \Z^2) \big)^M }
\end{align*}

\noi
for any $M \ge 1$. 


Regarding the variance, we use a Riemann sum approximation to obtain
\begin{align*}
\E \big[ |\Phi_{q,N}(x_0)|^2 \big]&=\sum_{ |n| \le N} \frac{|n|^4}{1+|n|^2} |\ft Q(q^{-1/2} n)|^2+O(e^{-cq^{\dl}})\\
&\sim q^2 \int_{\R^2} \frac{|\xi|^4}{1/q+|\xi|^2 } |\ft Q(\xi )|^2 d\xi \sim q^2
\end{align*}

\noi
as $q\to \infty$. This completes the proof of Proposition \ref{PROP:corr}.

\end{proof}

\subsection{Coarse graining and discretization}

In this subsection, we present a coarse-graining argument for the continuous Gaussian process $\{ \Phi_{q,N}(x_0)  \}_{x_0\in \T^2}$, based on the correlation decay estimate (Proposition \ref{PROP:corr}). Recall from Proposition \ref{PROP:corr} that
\begin{align*}
\textup{corr}(\Phi_{q,N}(x_0), \Phi_{q,N}(x_1) ) \les \frac{1}{\big( 1+q^{\frac 12}  \textup{dist}(x_0-x_1, 2\pi \Z^2) \big)^M }
\end{align*}

\noi
for any $M \ge 1$. Therefore, the correlation becomes negligible as $q\to \infty$ once the spatial distance $\textup{dist}(x_0-x_1, 2\pi \Z^2)$ exceeds $q^{-\frac 12}$, or more precisely, $q^{-\frac 12}(\log q)^{\frac 12-\eps}$ for some small 
$\eps>0$. We thus identify the correlation length scale as
\begin{align}
\l_q:=q^{-\frac 12}.
\label{CO2}
\end{align}

\noi
We partition the torus $\T^2=(\R /  2\pi \Z)^2$ into a regular grid of squares of side length $\sim \dl_q$
\begin{align}
\dl_q=q^{-\frac 12}(\log q)^{\frac 12-\eps}
\label{CO3}
\end{align}

\noi
in accordance with the correlation length scale $q^{-\frac 12}$ of the Gaussian field  $\Phi_{q,N}(x)$.   Let $\Ld_q$ denote the collection of center points of these squares. In particular, the total number of such boxes (or equivalently, the number of points in $\Lambda_q$) satisfies
\begin{align*}
\# \Ld_q \sim \Big( \frac {2\pi}{\dl_q}   \Big)^2 \sim q(\log q)^{-1+2\eps}.
\end{align*}

\noi
Since $x_{j} \neq x_k \in \Ld_q$ are \underline{centers} of boxes in a partition of the torus $\T^2$ into square boxes of side length $\sim \dl_q=q^{-\frac 12}(\log q)^{\frac 12-\eps}$, we have 
\begin{align}
x_j-x_k \notin 2\pi \Z^2.
\label{CO1}
\end{align}

\noi
Using the grid spacing $\dl_q= q^{-\frac 12}(\log q)^{\frac 12-\eps}$ and \eqref{CO1}, we have
\begin{align}
q^{\frac 12}  \text{dist}(x_{j}-x_{k}, 2\pi \Z^2) \ges q^{\frac 12} \dl_q=(\log q)^{\frac 12-\eps}.
\end{align}

\noi
This implies that for any distinct center points $x_j\neq x_k \in \Ld_q$,
\begin{align*}
\textup{corr}(\Phi_{q,N}(x_0), \Phi_{q,N}(x_1) ) \les (\log q)^{-\frac M2+\eps M} \to 0
\end{align*}

\noi
as $q\to \infty$. Therefore, thanks to the coarse graining, the discretized fields $\Phi_{q,N}(x_j)$, indexed by the center points $x_j \in \Lambda_q$ of the boxes, are weakly correlated across different boxes. This allows us to treat the contributions from distinct boxes as approximately independent in the limit $q\to \infty$.

In summary, we obtain a family of discretized Gaussian fields, indexed by the center points:
\begin{align*}
\{ \Phi_{q,N}(x_j) \}_{j\in \Ld_q}
\end{align*}

\noi
and observe the following:
\begin{itemize}

\item [(1)] Each box has diameter $\sim \dl_q=q^{-\frac 12}(\log q)^{\frac 12-\eps}$

\smallskip 

\item [(2)] The distance between centers of distinct boxes satisfies $\ges \dl_q=q^{-\frac 12}(\log q)^{\frac 12-\eps}$.

\smallskip 

\item[(3)] Points within a single box may still have non-negligible correlation.

\smallskip 

\item[(4)] However, the center points $x_j \in \Ld_q$ from different boxes are separated by more than the correlation length $\ell_q=q^{-\frac 12}$ from \eqref{CO2}, so their correlations decay in $q$, as shown in \eqref{CO3}.

\smallskip

\item[(5)] Since the centers $x_j \neq x_k \in \Ld_q$ arise from partitioning $\T^2$ into square boxes of side length $\dl_q$, they are distinct points on the torus, that is, $x_j-x_k \notin 2\pi \Z^2$.

\end{itemize}
















\subsection{Discretized approximation of the maximum of the Gaussian Process}

In this subsection, we study the discretized approximation of 
the Gaussian process $\{\Phi_{q,N}(x_0)\}_{x_0 \in \T^2}$.  Under the choice of coarse-graining scale $\dl_q=q^{-\frac 12}(\log q)^{\frac 12 -\eps}$, the following proposition shows that the maximum of the continuous Gaussian process is well approximated by that of its discretized version.

\begin{proposition}\label{PROP:disapp}
Let $\Phi_{q,N}(x_0)$ be the Gaussian process over $x_0 \in \T^2$, as defined in Proposition~\ref{PROP:fluc}. Then, by choosing $N=N(q)=q^{\frac 52+}$, as in Proposition \ref{PROP:fluc}, we obtain 
\begin{align*}
\E\bigg[\max_{ x \in \T^2 } \Phi_{q,N}(x) \bigg]
=\E\bigg[\max_{x_{j}  \in \Ld_q } \Phi_{q,N}(x_{j}) \bigg]+o(q\sqrt{\log q}).
\end{align*}

\noi
as $ q \to \infty$, where $\Ld_q$ is the collection of center points obtained by partitioning the torus $\T^2$ into square boxes of side length $\sim \dl_q=q^{-\frac 12} (\log q)^{\frac 12-\eps}$.




\end{proposition}

\begin{remark}\rm 
In the next subsection (Proposition \ref{PROP:max}), we prove that the leading-order term satisfies
\begin{align*}
\E\bigg[\max_{x_{j}  \in \Ld_q } \Phi_{q,N}(x_{j}) \bigg] \sim q \sqrt{\log q}
\end{align*}

\noi
as $q\to \infty$.  Accordingly, in Proposition \ref{PROP:disapp}, we show that the error term is of lower order, that is, $o(q\sqrt{\log q})$. Therefore, the discretized approximation accurately captures the essential behavior of the continuous field
\begin{align*}
\E\bigg[\max_{ x \in \T^2 } \Phi_{q,N}(x) \bigg]
\sim \E\bigg[\max_{x_{j}  \in \Ld_q } \Phi_{q,N}(x_{j}) \bigg]
\end{align*}

\noi
as $q\to \infty$.  We point out that our choice of coarse-graining scale $\dl_q=q^{-\frac 12}(\log q)^{\frac 12 -\eps}$ is sufficient to ensure that the error term is negligible compared to the leading order. See \eqref{J14}.
\end{remark}

Before proving the discretized approximation of the maximum of the continuous Gaussian process $\Phi_{q, N}$, we state Dudley’s inequality, which plays a crucial role in the argument.

\begin{lemma}[Dudley's entropy inequality]\label{LEM:Dud}
Let $\{X_t: t\in T \}$ be a centered Gaussian process equipped with the canonical metric
\begin{align*}
d(s,t):=\big(\E|X_s-X_t|^2 \big)^\frac 12,
\end{align*}

\noi
and let 
\begin{align*}
\textup{diam}(T):=\sup_{s,t \in T} d(s,t).
\end{align*}

\noi
Then, we have the bound
\begin{align*}
\E\bigg[ \sup_{t \in T} X_t \bigg] \les \int_0^{\textup{diam}(T) } \sqrt{\log   N(T,d,\eps) } d\eps,
\end{align*}

\noi
where $N(T,d,\eps)$ is the minimal number of $d$-balls of raduis $\eps$ needed to cover $T$,  known as the entropy number. 

\end{lemma}

We are now ready to prove Proposition \ref{PROP:disapp}.


\begin{proof}[Proof of Proposition \ref{PROP:disapp}]
For each $x \in \T^2$, there exists a  point $x_j \in \Ld_q$ such that $|x-x_j| \les \dl_q$ and 
\begin{align*}
\Phi_{q,N}(x) \le \max_{x_{j} \in \Ld_q } \Phi_{q,N}(x_{j})+\sup_{ \substack{ y,z\in \T^2 \\  |y-z|\les \dl_q  }} |\Phi_{q,N}(y)-\Phi_{q,N}(z) |.
\end{align*}

\noi
Taking the maximum over $x\in \T^2$ and then the expectation,
\begin{align}
\E\bigg[ \max_{x\in \T^2} \Phi_{q,N}(x) \bigg] \le \E\bigg[ \max_{x_{j }\in \Ld_q} \Phi_{q,N}(x_{j}) \bigg] +\E\Bigg[ \sup_{\substack{y, z \in \T^2 \\|y-z|\les \dl_q}}|\Phi_{q,N}(y)-\Phi_{q,N}(z)|  \Bigg].
\label{J0} 
\end{align}

\noi
Define the process (two-parameter family)
\begin{align}
\Psi_{q,N}(y,z):=\Phi_{q,N}(y)-\Phi_{q,N}(z)
\label{J4}
\end{align}

\noi
over the set $\mathcal{D}_{\dl_q}:=\{(y,z)\in \T^2 \times \T^2: |y-z| \les \dl_q \}$. Notice that $\Psi_{q,N}(y,z)$  is a centered Gaussian process, since the family
 $\{ \Phi_{q,N}(x_0) \}_{x_0 \in \T^2}$ is jointly Gaussian and stationary. This follows from \eqref{D00}, which gives
\begin{align*}
\Phi_{q,N}=\sum_{|n| \le N} a_n(x_0) B_n(\tfrac 12,\o)
\end{align*}

\noi
where $a_n(x_0):=\frac{|n|^2}{\sqrt{1+|n|^2}}  \int_{\T^2}  Q_{q,x_0}(x)e^{in \cdot x}dx$ and $\{B_n\}_{n \in \Z^2}$ is a family of independent Browninan motions. Hence, we apply Dudley's inequality (Lemma \ref{LEM:Dud}) to control 
\begin{align*}
\E\bigg[  \sup_{(y,z)\in \mathcal{D}_{\dl_q} } |\Psi_{q,N }(y,z) |  \bigg].
\end{align*}

\noi
Define the canonical metric on $\T^2 \times \T^2$
\begin{align}
d\big( (y,z), (y',z') \big):= \Big( \E \big[   |  \Psi(y,z)-\Psi(y',z') |^2  \big] \Big)^{\frac 12},
\label{J5}
\end{align}

\noi
where $(y,z), (y',z') \in \T^2\times \T^2$. By the definition \eqref{J4}, 
\begin{align}
\E \big[   |  \Psi(y,z)-\Psi(y',z') |^2  \big]  \les \E\big[|\Phi_{q,N}(y)-\Phi_{q,N}(z)|^2 \big] +\E\big[|\Phi_{q,N}(y')-\Phi_{q,N}(z')|^2 \big].
\label{J6}
\end{align}

From \eqref{DD22}, we write 
\begin{align*}
\Phi_{q,N}(y)-\Phi_{q,N}(z)=\sum_{|n|\le N}\frac{|n|^2}{\sqrt{1+|n|^2} } B_n(\tfrac 12,\o)  \cdot   \ft Q(q^{-\frac 12} n ) \cdot (e^{in\cdot y} -e^{in \cdot z}).
\end{align*}

\noi
Thanks to the independence of the $B_n$, we have 
\begin{align*}
\E\big[|\Phi_{q,N}(y)-\Phi_{q,N}(z)|^2 \big]=\frac 12 \cdot \sum_{|n| \le N} \frac{|n|^4}{1+|n|^2 } | \ft Q(q^{-\frac 12} n ) |^2 |e^{in\cdot y}-e^{in \cdot z} |^2.
\end{align*}

\noi
This implies that for any $y,z \in \T^2$, 
\begin{align}
\E\big[|\Phi_{q,N}(y)-\Phi_{q,N}(z)|^2 \big] \les |y-z|^2 \sum_{|n| \le N} |n|^4 | \ft Q(q^{-\frac 12} n ) |^2 \les q^2 |y-z|^2,
\label{J7}
\end{align}

\noi
where we used  the Riemann approximation 
\begin{align*}
\sum_{n \in \Z^2} |n|^4 | \ft Q(q^{-\frac 12} n ) |^2 \sim q^3 \int_{\R^2} |\xi|^2 |\ft Q(\xi)|^2 d \xi \sim q^3.
\end{align*}

\noi
as $q\to \infty$. Combining \eqref{J5}, \eqref{J6}, and \eqref{J7} yields that the canonical metric \eqref{J5} satisfies 
\begin{align}
d\big( (y,z), (y',z') \big) \les q^{\frac 32} |y-z|+q^\frac 32|y'-z'| \les q^\frac 32 \dl_q,
\label{J8} 
\end{align}

\noi
where we used the condition $|y-z| \les \dl_q$, valid over the set  $\mathcal{D}_{\dl_q}:=\{(y,z)\in \T^2 \times \T^2: |y-z| \les \dl_q \}$.

We are now ready to apply Dudley's inequality (Lemma \ref{LEM:Dud}). Under the condition \eqref{J8} $d\big( (y,z), (y',z') \big) \les  q^{\frac 32} \dl_q$, where $(y,z), (y',z') \in \mathcal{D}_{\dl_q}$, the number of $\eps$-balls needed to cover the set $\mathcal{D}_{\dl_q}$ is 
\begin{align}
N(\mathcal{D}_{\dl_q}, 4,  \eps)  \les \bigg( \frac{ q^{\frac 32} \dl_q }{\eps}\bigg)^4.
\label{J9}
\end{align}

\noi
It follows from \eqref{J4}, Dudley's inequality (Lemma \ref{LEM:Dud}) and \eqref{J9} that 
\begin{align}
\E\Bigg[ \sup_{\substack{y, z \in \T^2 \\|y-z|\le \dl_q}}|\Phi_{q,N}(y)-\Phi_{q,N}(z)|  \Bigg]=\E\bigg[  \sup_{(y,z)\in \mathcal{D}_{\dl_q} } |\Psi_{q,N }(y,z) |  \bigg] \les \int_0^{ q^{\frac 32}  \dl_q } \sqrt{\log \Big( \frac{  q^{\frac 32}    \dl_q }{\eps}  \Big) } d\eps. 
\label{J10} 
\end{align}

\noi
Taking the change of variable $u=\frac{\eps}{ q^{\frac 32}  \dl_q  }$ yields 
\begin{align}
\int_0^{ q^{\frac 32}  \dl_q } \sqrt{\log \Big( \frac{q^{\frac 32}  \dl_q }{\eps}  \Big) } d\eps=q^{\frac 32}  \dl_q \int_0^1  \sqrt{\log \frac 1u  } du \sim  q^{\frac 32}  \dl_q
\label{J11}
\end{align}

\noi
since $\int_0^1  \sqrt{\log \frac 1u  } du < \infty$. Therefore, from \eqref{J0}, \eqref{J10}, and \eqref{J11}, we obtain  
\begin{align}
\E\bigg[ \max_{x\in \T^2} \Phi_{q,N}(x) \bigg] \le \E\bigg[ \max_{x_{j }\in \Ld_q} \Phi_{q,N}(x_{j}) \bigg]+C \cdot q^{\frac 32}  \dl_{q} 
\label{J12}
\end{align}

\noi
for some constant $C>0$. By the definition of the maximum, we have 
\begin{align}
\E\bigg[ \max_{x\in \T^2} \Phi_{q,N}(x) \bigg] \ge \E\bigg[ \max_{x_{j }\in \Ld_q} \Phi_{q,N}(x_{j}) \bigg].
\label{J13}
\end{align}

\noi
Since the grid has spacing $\dl_q:=q^{-\frac 12} (\log q)^{\frac 12 -\eps}$,  it follows from \eqref{J12} and \eqref{J13} that
\begin{align}
\E\bigg[ \max_{x\in \T^2} \Phi_{q,N}(x) \bigg]
&\sim \E\bigg[ \max_{x_{j }\in \Ld_q} \Phi_{q,N}(x_{j}) \bigg]+{q^{\frac 32} } \dl_q \notag  \\
&\sim \E\bigg[ \max_{x_{j }\in \Ld_q} \Phi_{q,N}(x_{j}) \bigg]+{q^{\frac 32} }    \cdot { q^{-\frac 12} (\log q)^{\frac 12 -\eps} }  \notag \\
&\sim \E\bigg[ \max_{x_{j }\in \Ld_q} \Phi_{q,N}(x_{j}) \bigg]+o(q\sqrt{\log q})
\label{J14}
\end{align}

\noi
as $q\to \infty$. This completes the proof of Proposition \ref{PROP:disapp}.

\end{proof}

\subsection{Maximum of discretized Gaussian processes}

In Proposition \ref{PROP:disapp}, we show that the discretized version provides a good approximation of the maximum of the continuous Gaussian process
\begin{align*}
\E\bigg[\max_{ x \in \T^2 } \Phi_{q,N}(x) \bigg]
=\E\bigg[\max_{x_{j}  \in \Ld_q } \Phi_{q,N}(x_{j}) \bigg]+o(q\sqrt{\log q})
\end{align*}

\noi
as $q \to \infty$. In the following proposition, we analyze the maximum of the discretized Gaussian process.

\begin{proposition}\label{PROP:max} 

Let $\Phi_{q,N}(x_0)$ be the Gaussian process over $x_0 \in \T^2$, as defined in Proposition~\ref{PROP:fluc}. Then, by choosing $N=N(q)=q^{\frac 52+}$, as in Proposition \ref{PROP:fluc}, we obtain 
\begin{align*}
\E\bigg[\max_{x_{j}  \in \Ld_q } \Phi_{q,N}(x_{j}) \bigg]\sim q \sqrt{\log \# \Ld_q } \sim q \sqrt{\log q},
\end{align*}

\noi
where $\Ld_q$ is the collection of center points obtained by partitioning the torus $\T^2=(\R /  2\pi \Z)^2$ into square boxes of side length $\sim \dl_q=q^{-\frac 12} (\log q)^{\frac 12 -\eps}$. In particular,  the total number of center points 
\begin{align*}
\# \Ld_q \sim \Big( \frac {2\pi}{\dl_q}   \Big)^2 \sim q (\log q)^{-1+2\eps}.
\end{align*}

\end{proposition}

\begin{remark}\rm 

Note that the Gaussian fields $\Phi_{q,N}(x_j)$, $j\in \Ld_q$, are not independent. 
Therefore, the behavior of the maximum of the discretized Gaussian process is not as straightforward as in the case of independent Gaussian variables. In the following, we show that under the chosen coarse-graining scale $\dl_q=q^{-\frac 12} (\log q)^{\frac 12 -\eps}$, the discretized Gaussian fields are weakly correlated (Proposition \ref{PROP:corr}), allowing us to show that the weakly correlated Gaussian fields behave like independent ones in terms of their maxima.


\end{remark}

Before presenting the proof of Proposition \ref{PROP:max}, we introduce Sudakov's inequality(see \cite[Page 80]{Led}), which plays a key role in the argument.

\begin{lemma}[Sudakov inequality]\label{LEM:sud}
Let $\{X_t\}_{t\in T}$ be a centered Gaussian process. Then,
\begin{align*}
\E\Big[\sup_{t \in T} X_t \Big] \ges \inf_{s\neq t}d(s,t) \cdot \sqrt{\log |T|}.
\end{align*}

\noi
where
\begin{align*}
d(s,t)=\big(\E|X_s-X_t|^2 \big)^\frac 12.
\end{align*}

\end{lemma}

We are now ready to present the proof of Proposition \ref{PROP:max}.

\begin{proof}[Proof of Proposition \ref{PROP:max}]
We first show the lower bound
\begin{align*}
\E\bigg[\max_{x_{j}  \in \Ld_q } \Phi_{q,N}(x_{j}) \bigg] \ges q \sqrt{\log \# \Ld_q } 
\end{align*}

\noi 
For any $x_{j}\neq x_{k} \in \Ld_q$, using Proposition \ref{PROP:corr}, we have
\begin{align}
\E\big[|\Phi_{q,N}(x_{j} ) -\Phi_{q,N}(x_{k})|^2 \big]&= \E\big[ |\Phi_{q,N}(x_j)|^2 \big]+ \E\big[ |\Phi_{q,N}(x_k)|^2 \big]-2\E\big[\Phi_{q,N}(x_{j} ) \Phi(x_{k}) \big] \notag \\
&\sim  2q^2 -2\E\big[\Phi_{q,N}(x_{j} ) \Phi(x_{k}) \big] \notag \\
& \ges 2q^2- \frac{q^2}{\big( 1+q^{\frac 12}  \text{dist}(x_{j}-x_{k}, 2\pi \Z^2) \big)^M } 
\label{JJ15}
\end{align}

\noi
for any $M \ge 1$. Since $x_{j} \neq x_k \in \Ld_q$ are centers of boxes in a partition of the torus $\T^2$ into square boxes of side length $\sim \dl_q=q^{-\frac 12} (\log q)^{\frac 12 -\eps}$, we have 
\begin{align}
x_j-x_k \notin 2\pi \Z^2.
\label{J15}
\end{align}

\noi
Using the grid spacing $\dl_q=q^{-\frac 12} (\log q)^{\frac 12 -\eps}$ and \eqref{J15}, we have
\begin{align}
q^{\frac 12}  \text{dist}(x_{j}-x_{k}, 2\pi \Z^2) \ges q^{\frac 12} \dl_q=(\log q)^{\frac 12 -\eps}.
\label{J16}
\end{align}

\noi
Combining \eqref{JJ15} and \eqref{J16} yields 
\begin{align}
q \ges \inf_{\substack{ x_{j}, x_{k} \in \Ld_q \\ x_{j} \neq x_{k} } } \big(\E\big[|\Phi_{q,N}(x_{j} ) -\Phi_{q,N}(x_{k})|^2 \big] \big)^\frac 12  \ges \big( 2q^2-q^2(\log q)^{-\frac M2+\eps M}  \big)^\frac 12 \ges q,
\label{J17}
\end{align}

\noi
where the first upper bound is immediate, since $\E\big[ |\Phi_{q,N}(x_k)|^2 \big] \sim q^2$ by Proposition \ref{PROP:corr}. Thus, Sudakov’s inequality (Lemma \ref{LEM:sud}), together with \eqref{J17}, gives
\begin{align*}
\E\bigg[\max_{x_{j}  \in \Ld_q } \Phi_{q,N}(x_{j}) \bigg] \ges q \sqrt{\log \# \Ld_q }. 
\end{align*}

\noi

Regarding the upper bound, regardless of the covariance structure, for any collection of Gaussian random variables $X_i$, we have 
\begin{align*}
\E\Big[\max_{1 \le i \le J} X_i\Big] \le C \cdot \sqrt{\max_i \E[X_i^2]  } \cdot \sqrt{\log J}
\end{align*}

\noi
for some constant $C$ independent of $J \ge 1$. Therefore, we have 
\begin{align*}
\E\bigg[\max_{x_{j}  \in \Ld_q } \Phi_{q,N}(x_{j}) \bigg] \les q \sqrt{\log \# \Ld_q } 
\end{align*}

\noi 
since $\E\big[ |\Phi_{q,N}(x_k)|^2 \big] \sim q^2$ by Proposition \ref{PROP:corr}. This completes the proof of Proposition \ref{PROP:max}.


\end{proof}

\subsection{Proof of the critical case}

In this subsection, we present the proof of the main theorem (Theorem \ref{THM:1}) in the critical case $A=A_0$.

\begin{proof}
From Proposition \ref{PROP:fluc},
\begin{align}
\log Z_{A,N(q)} \ge \E\bigg[ \max_{x \in \T^2} \Phi_{q,N(q)}(x) \bigg] -q -e^{-cq }.
\label{J18}
\end{align}
\noi 
It follows from Propositions \ref{PROP:disapp} and \ref{PROP:max} that 
\begin{align}
\E\bigg[\max_{ x \in \T^2 } \Phi_{q,N(q)}(x) \bigg]&\sim \E\bigg[\max_{x_{j}  \in \Ld_q } \Phi_{q,N(q)}(x_{j}) \bigg]+o(q\sqrt{\log q})  \notag \\
&\sim q \sqrt{\log q}
\label{J19}
\end{align}

\noi
as $q \to \infty$. Combining \eqref{J18} and \eqref{J19} yields 
\begin{align*}
\log Z_{A,N(q)} \ges q\sqrt{\log q}-q-e^{-cq} \to \infty
\end{align*}

\noi
as $q \to \infty$. This concludes the proof of Theorem \ref{THM:1} in the critical case $A=A_0$.

\end{proof}

\begin{ackno}\rm
The work of P.S. is partially supported by NSF grants DMS-1811093,  DMS-2154090 and a Simons Fellowship.
\end{ackno}

\end{document}